\documentclass[11pt, twoside]{article}
\usepackage{amssymb, amsmath, amsthm}
\usepackage[left=20mm, right=20mm, bottom=20mm]{geometry}
\usepackage{inputenc}
\usepackage{fancyhdr}
\usepackage{tikz}
\usepackage{float}
\usetikzlibrary{positioning}
\usepackage{titling}
\usepackage{hyperref}
\predate{}
\postdate{}
\usepackage{lipsum}
\newtheorem{theorem}{Theorem}
\newtheorem{lemma}[theorem]{Lemma}
\newtheorem{proposition}[theorem]{Proposition}
\newtheorem{claim}{Claim}

\newtheorem{claim1}{Claim}
\newtheorem{question}[theorem]{Question}
\newtheorem{conjecture}[theorem]{Conjecture}
\usepackage{titling}
\usepackage{xcolor}
\usepackage{caption}
\usepackage{placeins}
\usepackage{tikz}
\usepackage{verbatim}
\usepackage{titling}
\usepackage[T1]{fontenc}
\usepackage{currvita}
\usepackage{bbm}
\usepackage{enumitem}
\newtheorem{corollary}[theorem]{Corollary}

\theoremstyle{definition}

\theoremstyle{remark}

\newcommand{\cF}{\mathcal{F}}

\newcommand{\cN}{\mathcal N}

\begin{document}
\newcommand{\Addresses}{
\bigskip
\footnotesize
		
\medskip

\noindent Maria-Romina~Ivan, \textsc{Department of Pure Mathematics and Mathematical Statistics, Centre for Mathematical Sciences, Wilberforce Road, Cambridge, CB3 0WB, UK,} and\\\textsc{Department of Mathematics, Stanford University, 450 Jane Stanford Way, CA 94304, USA.}\par\noindent\nopagebreak\textit{Email addresses: }\texttt{mri25@dpmms.cam.ac.uk, m.r.ivan@stanford.edu}
		
\medskip
		
\noindent Nandi~Wang, \textsc{Magdalene College, University of Cambridge, CB3 0AG, UK.}\par\noindent\nopagebreak\textit{Email address: }\texttt{nw443@cam.ac.uk}}

\pagestyle{fancy}
\fancyhf{}
\fancyhead [LE, RO] {\thepage}
\fancyhead [CE] {MARIA-ROMINA IVAN AND NANDI WANG}
\fancyhead [CO] {LINEAR SATURATION FOR $\mathcal N$ VIA BUTTERFLIES}
\renewcommand{\headrulewidth}{0pt}
\renewcommand{\l}{\rule{6em}{1pt}\ }
\title{\Large{\textbf{LINEAR SATURATION FOR $\mathcal N$ VIA BUTTERFLIES}}}
\author{MARIA-ROMINA IVAN AND NANDI WANG}
\date{ }
\maketitle
\begin{abstract}
Given a finite poset $\mathcal P$, how small can a family $\mathcal F$ of subsets of $[n]$ be such that $\mathcal F$ does not contain an induced copy of $\mathcal P$, but $\mathcal F\cup\{X\}$ contains such a copy for all $X\in\mathcal P([n])\setminus\mathcal F$? This is known as the induced saturation number of $\mathcal P$, denoted by $\text{sat}^*(n,\mathcal P)$. The main conjecture in the area is that the induced saturation number for any poset is either bounded, or linear.

In this paper we establish linearity for the induced saturation number of the 4-point poset $\mathcal N$. Previously, it was known that $2\sqrt n\leq\text{sat}^*(n,\mathcal N)\leq 2n$. We show that $\text{sat}^*(n,\mathcal N)\geq\frac{n+6}{4}$. A crucial role in the proof is played by a structural feature of $\mathcal N$-saturated families, namely that if the family contains two antichains, one completely above the other, then it must also contain a `middle' point -- greater than one antichain and less than the other.
\end{abstract}
	
\section{Introduction}
We say that a poset $(\mathcal Q, \leq')$ contains an \textit{induced copy} of a poset $(\mathcal P, \leq)$ if there exists an injective function $f:\mathcal P\rightarrow\mathcal Q$ such that $(\mathcal P, \leq)$ and $(f(\mathcal P), \leq')$ are isomorphic. For a fixed poset $\mathcal P$ we call a family $\mathcal F$ of subsets of $[n]=\{1,2,\dots,n\}$ $\mathcal P$-\textit{saturated} if $\mathcal F$ does not contain an induced copy of $\mathcal P$, but for every subset $S$ of $[n]$ such that $S\notin\mathcal F$, the family $\mathcal F\cup \{S\}$ does contain such a copy. We denote by $\text{sat}^*(n, \mathcal P)$ the size of the smallest $\mathcal P$-saturated family of subsets of $[n]$. In general we refer to $\text{sat}^*(n, \mathcal P)$ as the \textit{induced saturation number} of $\mathcal P$.

This notion, inspired by saturation for graphs, was introduced by Ferrara, Kay, Kramer, Martin, Reiniger, Smith and Sullivan \cite{ferrara2017saturation}. Despite their similarity, graph saturation and poset saturation are vastly different, partially due to the rigidity of \textit{induced} copies of posets. One of the most important conjectures in the area is that for every poset the saturation number is either bounded or linear \cite{keszegh2021induced}. The saturation number for posets has already been shown to display a dichotomy. Keszegh, Lemons, Martin, P\'alv\"olgy and Patk\'os \cite{keszegh2021induced} showed that the saturation number is either bounded or at least $\log_2(n)$. This was later improved by Freschi, Piga, Sharifzadeh and Treglown \cite{freschi2023induced} who showed that $\text{sat}^*(n,\mathcal P)$ is either bounded or at least $2\sqrt n$. In the other direction, Bastide, Groenland, Ivan and Johnston \cite{polynomial} showed that the saturation number of any poset grows at most polynomially. Moreover, very recently Ivan and Jaffe \cite{gluing} showed that for any poset, one can add at most 3 points to obtain a poset with saturation number at most linear.  

Even more surprisingly, establishing the rate of growth even for simple posets is difficult, with only a handful of examples for which linearity or boundedness has been proven. Most notably, the following posets have linear saturation number: the butterfly (two maximal elements completely above two minimal elements) \cite{ivan2020saturationbutterflyposet, keszegh2021induced}, the antichain \cite{bastide2024exact}, and most recently the diamond (one maximal, one minimal and two other incomparable elements) \cite{diamondlinear}, whose saturation number was stuck for a long time at the $\sqrt n$ lower bound. 

In this paper we establish linearity for another key 4-poset, the $\mathcal N$, comprised of 2 minimal and 2 maximal elements, such that exactly one maximal element is comparable to both minimal elements, and exactly one minimal element is comparable to both maximal elements, as depicted below.
\begin{center}
\begin{figure}[h]
\centering
\hspace{0.1cm}\\
\begin{tikzpicture}
\node (top1) at (3,1) {$\bullet$}; 
\node (bottom1) at (3,-1) {$\bullet$};
\node (top2) at (5,1) {$\bullet$};
\node (bottom2) at (5,-1) {$\bullet$};
\draw (bottom1) -- (top1) -- (bottom2) -- (top2);
\end{tikzpicture}
\end{figure}
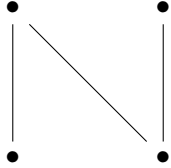
The Hasse diagram of the poset $\mathcal N$
\end{center}
The Hasse diagram of $\mathcal N$ is misleadingly simple. One of the main challenges when it comes to analysing this structure comes from the `broken' symmetry of the $\mathcal N$ -- some of its elements are comparable to exactly one other, and others are comparable to 2 -- unlike for example the very symmetric butterfly poset. 

It is not hard to check that the family $\mathcal F=\{\emptyset, [n], \{i\}, \{1,2,\dots,i\}:1\leq i\leq n\}$ is $\mathcal N$-saturated for all $n\geq 3$, hence $sat^*(n,\mathcal N)\leq 2n$ \cite{ferrara2017saturation}. However, the lower bound proved much more challenging. It was first shown that $\log_2n\leq\text{sat}^*(n,\mathcal N)$ \cite{ferrara2017saturation}, which was then improved to $\sqrt n\leq\text{sat}^*(n,\mathcal N)$ \cite{ivan2020saturationbutterflyposet}, and most recently to $2\sqrt n\leq\text{sat}^*(n,\mathcal N)$ \cite{freschi2023induced}. We mention that these bounds were not specific to the $\mathcal N$ poset -- they were in fact a consequence of the fact that the $\mathcal N$ is part of a larger class of special posets, namely posets that have the \textit{unique cover twin property} (UCTP).

In this paper we show that the size of any $\mathcal N$-saturated family with ground set $[n]$ is at least $\frac{n+6}{4}$, therefore showing that $\text{sat}^*(n,\mathcal N)=\Theta(n)$. This also concludes the analysis of all posets discussed in the paper that launched the field \cite{ferrara2017saturation}, as the $\mathcal N$ poset was the final missing piece. 

Our proof is heavily based on Lemma~\ref{lem:complete bip graph has a midpoint}, which asserts the following.
\setcounter{theorem}{2}
\begin{lemma}
Let $\mathcal F$ be an $\mathcal N$-saturated family, and $A_1, \dots, A_k$ and  $B_1, \dots, B_l$ sets in $\mathcal F$ such that $\cup_{i=1}^l B_i \subseteq \cap_{i=1}^k A_i$, where $k,l\geq 1$. Then, there exists $ M \in \mathcal F$ such that $\cup_{i=1}^l B_i \subseteq M \subseteq \cap_{i=1}^k A_i$.
\end{lemma}
\setcounter{theorem}{0}
In particular, if $\mathcal F$ contains a butterfly, it must also contain a `middle' point, less than its maximal elements and bigger than its minimal elements. This is crucial to our analysis as the existence of this midpoint will serve as a pivot for identifying forbidden copies of  $\mathcal N$.

The paper is organised as follows. Section 2 is dedicated to establishing the above mentioned connection between $\mathcal N$-saturated families and butterflies. In Section 3 we lay the groundwork for the main result by proving a few general lemmas that will be used frequently in the subsequent analysis. All these results come together in Section 4, Proposition~\ref{prop:main result}, from which the linear lower bound for $\text{sat}^*(n,\mathcal N)$ is then easily obtained. 

A great number of our lemmas are accompanied by diagrams -- however, as the analysis does not exhaustively investigate the relationship between each pair of elements, the lack of an edge does not immediately imply incomparability.

\section{$\mathcal N$-saturated families and butterflies}
In this section we analyse the structure of an $\mathcal N$-saturated family. More precisely, we show that in such a family, if there exist two antichains one above the other, then there exists a set in the family that lies between them. This will motivate our analysis in later sections.

Let $\mathcal F$ be an $\mathcal N$-saturated family with ground set $[n]$, where $n\geq3$. We observe that we must have $\emptyset, [n] \in \mathcal F$ as they are comparable to everything and therefore cannot form an $\mathcal N$ when added to the family.

Consider the Hasse diagram of $\mathcal F\setminus\{\emptyset, [n]\}$ as a graph, and let $\mathcal G$ be a connected component. We call $\mathcal G$ a \textit{component of $\mathcal F$} for simplicity.
	
We start with the following observation about maximal and minimal elements of $\mathcal G$.
\begin{lemma}\label{lem:max min in same comp are comparable}
Let $T$ be a maximal element and $S$ a minimal element of $\mathcal G$. Then $S \subseteq T$.
\end{lemma} 
\begin{proof}
If $S=T$, there is nothing to prove, so we may assume that $S\neq T$. Next, we add to the Hasse diagram of $\mathcal G$ all edges $XY$ if $X\subset Y$. In other words, we take the transitive closure of the Hasse diagram. In this new graph, by connectedness, there exists a shortest path from $S$ to $T$, say $S=X_1,X_2,\dots, X_l=T$, where $X_i\in \mathcal F\setminus\{\emptyset, [n]\}$ are pairwise distinct. If $l=2$, we are done. If $l=3$, we must have $S \subset X_2 \subset T$  since $S$ is a minimal element and $T$ is a maximal element, but then $S\subset T$ is a shorter path, a contradiction. Hence we may assume $l \geq 4$.

We regard this path as a directed one, with an arrow from $X_i$ to $X_{i+1}$ if $X_i\subset X_{i+1}$, and from $X_{i+1}$ to $X_i$ otherwise. By minimality of the path, the edges must alternate in direction, otherwise by transitivity there are some non-adjacent elements $X_i, X_j$ comparable, and replacing the edges $X_iX_{i+1}, \dots, X_{j-1}X_j$ by $X_iX_j$ gives a shorter path. Thus we have $S \subset X_2, X_2 \supset X_3, X_3 \subset X_4$, where $S$ is incomparable to both $X_3$ and $X_4$, and $X_2$ is incomparable to $X_4$. Therefore $S, X_2, X_3, X_4$ form an $\cN$ in $\cF$, a contradiction.
\end{proof}
The above lemma tells us that the maximal and minimal elements of a component $\mathcal G$ form a complete bipartite poset. Lemma~\ref{lem:complete bip graph has a midpoint} unifies these elements by identifying a set in $\mathcal F$ comparable to all maximal and all minimal elements of $\mathcal G$. The next result is an intermediary step towards that. 
\begin{lemma}\label{lem:positions of intersections/unions in N}
Let $A_1, \dots, A_k$ and $B_1, \dots, B_l$ be elements of $\mathcal F$, where $k,l\geq1$. If $\cup_{i=1}^{l} B_i \notin \mathcal F$, then $\cup_{i=1}^l B_i$ must be a maximal element of any $\mathcal N$ formed when added to $\cF$. Similarly, if $\cap_{i=1}^{k} A_i \notin \mathcal F$, then $\cap_{i=1}^{k}A_i$ must be a minimal element of any $\mathcal N$ formed when added to $\mathcal F$. Furthermore, $\cup_{i=1}^l B_i$ can be assumed to be the maximal element of some $\mathcal N$, that is comparable to both minimal elements, and $\cap_{i=1}^k A_i$ can be assumed to be the minimal element of  some $\mathcal N$, that is comparable to both maximal elements, as illustrated below.
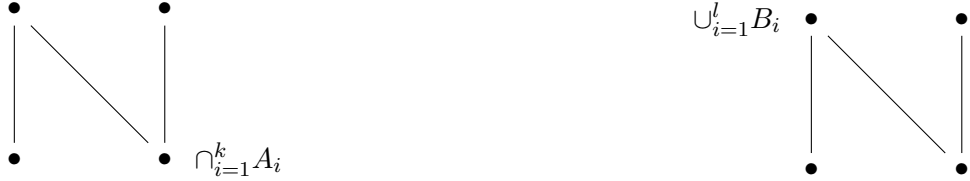
\begin{figure}[H]
\begin{minipage}{0.45\textwidth}
\centering
\begin{tikzpicture}[scale=1]
\node at (0,2) (Y) {\textbf{$\bullet$}};
\node at (2,2) (X) {\textbf{$\bullet$}};
\node at (0,0) (Z) {\textbf{$\bullet$}};
\node at (2,0) (U) {\textbf{$\bullet$}};
\node[left=1pt of Y] {$\cup_{i=1}^l B_i$};
\draw (Z) -- (Y);
\draw (X) -- (U);
\draw (U) -- (Y);
\end{tikzpicture}
\end{minipage}
\hfill
\begin{minipage}{0.45\textwidth}
\centering
\begin{tikzpicture}[scale=1]
\node at (0,2) (X) {\textbf{$\bullet$}};
\node at (2,2) (U) {\textbf{$\bullet$}};
\node at (0,0) (Y) {\textbf{$\bullet$}};
\node at (2,0) (Z) {\textbf{$\bullet$}};
\node[right=1pt of Z]{$\cap_{i=1}^k A_i$};

\draw (X) -- (Y);
\draw (X) -- (Z);
\draw (U) -- (Z);
\end{tikzpicture}
\end{minipage}
\end{figure}
\end{lemma}
\begin{proof} By the symmetry of both the statements and the structure of $\mathcal N$ (under taking complements), it is enough to prove the results for $\cup_{i=1}^l B_i$.
		
Suppose that $\cup_{i=1}^l B_i$ is a minimal element of an $\mathcal N$ formed with 3 other elements of $\mathcal F$. Then it is either the minimal element comparable with both maximal elements (figure (i) below), or it is the minimal element comparable to exactly one maximal element (figure (ii) below). Let $X$ denote the maximal element comparable with both minimal elements.
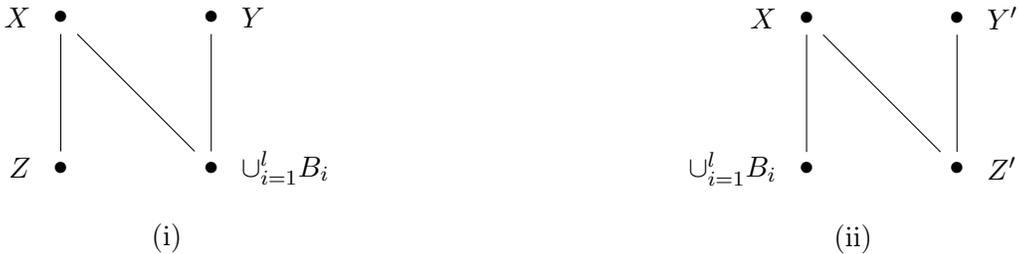
\begin{figure}[H]
\begin{minipage}{0.45\textwidth}
\centering
\begin{tikzpicture}[scale=1]
\node at (0,2) (Y) {\textbf{$\bullet$}};
\node at (2,2) (X) {\textbf{$\bullet$}};
\node at (0,0) (Z) {\textbf{$\bullet$}};
\node at (2,0) (U) {\textbf{$\bullet$}};
\node[right=1pt of U] {$\cup_{i=1}^l B_i$};
\node[left=1pt of Z] {$Z$};
\node[left=1pt of Y] {$X$};
\node[right=1pt of X] {$Y$};
\draw (Z) -- (Y);
\draw (X) -- (U);
\draw (U) -- (Y);
\end{tikzpicture}
\caption*{(i)}
\end{minipage}
\hfill
\begin{minipage}{0.45\textwidth}
\centering
\begin{tikzpicture}[scale=1]
\node at (0,2) (X) {\textbf{$\bullet$}};
\node at (2,2) (U) {\textbf{$\bullet$}};
\node at (0,0) (Y) {\textbf{$\bullet$}};
\node at (2,0) (Z) {\textbf{$\bullet$}};
\node[left=1pt of Y] {$\cup_{i=1}^l B_i$};
\node[right=1pt of Z] {$Z'$};
\node[right=1pt of U] {$Y'$};
\node[left=1pt of X] {$X$};
\draw (X) -- (Y);
\draw (X) -- (Z);
\draw (U) -- (Z);
\end{tikzpicture}
\caption*{(ii)}
\end{minipage}
\end{figure}

In the first case, since $\cup_{i=1}^l B_i$ is incomparable to $Z$, $Z\not\subseteq B_i$ for all $1\leq i\leq l$. Moreover, there exists $B_j \not\subset Z$ for some $1\leq j\leq l$. This implies that $Z$, $X$, $Y$ and $B_j$ form an $\mathcal N$ in $\mathcal F$, a contradiction.
		
In the second case, since $\cup_{i=1}^lB_i$ is incomparable to $Y'$, then $Y'\not\subseteq B_i$ for all $1\leq i\leq l$, and there exists $B_s$ such that $B_s\not\subset Y'$ for some $1\leq s\leq l$. Since $B_s$ and $Y'$ are incomparable, $B_s\not\subseteq Z'$, and since $Z'$ and $\cup_{i=1}^lB_i$ are incomparable, $Z'\not\subseteq B_s$. This means that $Z'$ and $B_s$ are incomparable, which yields a copy of $\mathcal N$ in $\mathcal F$, given by $B_s, X, Z'$ and $Y'$. 

This finishes the first part of the lemma.

Therefore, if $\cup_{i=1}^lB_i\notin\mathcal F$, it can only be a maximal element of any $\mathcal N$ formed when added to the family, as shown in below.
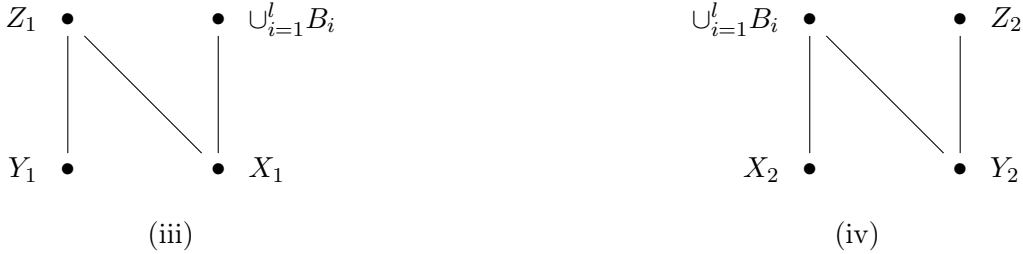
\begin{figure}[H]
\begin{minipage}{0.45\textwidth}
\centering
\begin{tikzpicture}[scale=1]
\node at (0,2) (X) {\textbf{$\bullet$}};
\node at (2,2) (U) {\textbf{$\bullet$}};
\node at (0,0) (Y) {\textbf{$\bullet$}};
\node at (2,0) (Z) {\textbf{$\bullet$}};
\node[left=1pt of X] {$Z_1$};
\node[left=1pt of Y] {$Y_1$};
\node[right=1pt of U] {$\cup_{i=1}^l B_i$};
\node[right=1pt of Z] {$X_1$};
\draw (X) -- (Y);
\draw (X) -- (Z);
\draw (U) -- (Z);
\end{tikzpicture}
\caption*{(iii)}
\end{minipage}
\hfill
\begin{minipage}{0.45\textwidth}
\centering
\begin{tikzpicture}[scale=1]
\node at (0,2) (X) {\textbf{$\bullet$}};
\node at (2,2) (U) {\textbf{$\bullet$}};
\node at (0,0) (Y) {\textbf{$\bullet$}};
\node at (2,0) (Z) {\textbf{$\bullet$}};
\node[left=1pt of X] {$\cup_{i=1}^l B_i$};
\node[left=1pt of Y] {$X_2$};
\node[right=1pt of U] {$Z_2$};
\node[right=1pt of Z] {$Y_2$};
\draw (X) -- (Y);
\draw (X) -- (Z);
\draw (U) -- (Z);
\end{tikzpicture}
\caption*{(iv)}
\end{minipage}
\end{figure}

To finish the claim, we show that in fact, we may always assume that $\cup_{i=1}^lB_i$ is the maximal element comparable to both minimal elements (figure (iv)).	

Suppose that it is not, and therefore $\cup_{i=1}^lB_i$ forms an $\mathcal N$ with $X_1, Y_1, Z_1\in\mathcal F$, as shown in figure (iii). Since $\cup_{i=1}^{l} B_i$ is incomparable to $Z_1$, as before, there exists $B_j$ incomparable to $Z_1$, for some $1\leq j\leq l$. Furthermore, since $Y_1 \subset Z_1$, $B_j \not \subseteq Y_1$, and since $\cup_{i=1}^l B_i$ is incomparable to $Y_1$ we also get that $Y_1 \not \subset B_j$. Therefore $B_j$ and $Y_1$ are incomparable. We also note that $B_j \not \subseteq X_1\subset Z_1$, as $B_j$ is incomparable to $Z_1$.
		
To avoid $B_j, X_1, Y_1, Z_1$ forming an $\mathcal N$ in $\mathcal F$, we must have $X_1 \not \subset B_j$. Therefore, $B_j$ and $X_1$ are incomparable. This implies that $X_1, Z_1, \cup_{i=1}^l B_i, B_j$ form an $\mathcal N$ where $\cup_{i=1}^l B_i$ is the maximal comparable to both minimal elements (as shown in figure (iv)). This finishes the claim.
\end{proof}
We are now ready to prove the main result of this section.	
\begin{lemma}\label{lem:complete bip graph has a midpoint}
Let $A_1, \dots, A_k$ and  $B_1, \dots, B_l$ be sets in $\mathcal F$ such that $\cup_{i=1}^l B_i \subseteq \cap_{i=1}^k A_i$, where $k,l\geq 1$. Then there exists $ M \in \mathcal F$ such that $\cup_{i=1}^l B_i \subseteq M \subseteq \cap_{i=1}^k A_i$, as illustrated below.
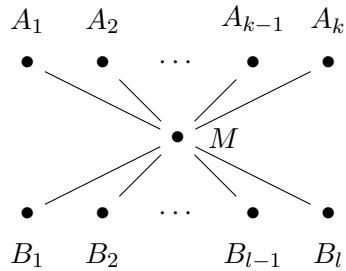
\begin{center}
\begin{tikzpicture}[scale=1]
\node at (0,0) (B1) {\textbf{$\bullet$}};
\node at (0,2) (A1) {\textbf{$\bullet$}};
\node at (1,2) (A2) {\textbf{$\bullet$}};
\node at (3,2) (Akm) {\textbf{$\bullet$}};
\node at (4,2) (Ak) {\textbf{$\bullet$}};
\node at (1,0) (B2) {\textbf{$\bullet$}};
\node at (3,0) (Blm) {\textbf{$\bullet$}};
\node at (4,0) (Bl) {\textbf{$\bullet$}};
\node at (2,1) (M) {\textbf{$\bullet$}};
\node at (2,0) {$\cdots$};
\node at (2,2) {$\cdots$};
\node[above=1pt of A1] {$A_1$};
\node[above=1pt of A2] {$A_2$};
\node[above=1pt of Ak] {$A_k$};
\node[above=1pt of Akm] {$A_{k-1}$};
\node[below=1pt of B1] {$B_1$};
\node[below=1pt of B2] {$B_2$};
\node[below=1pt of Blm] {$B_{l-1}$};
\node[below=1pt of Bl] {$B_l$};
\node[right=1pt of M] {$M$};
\draw (A1) -- (M);
\draw (A2) -- (M);
\draw (Ak) -- (M);
\draw (Akm) -- (M);						
\draw (B1) -- (M);
\draw (B2) -- (M);
\draw (Bl) -- (M);
\draw (Blm) -- (M);
\end{tikzpicture}
\end{center}
\end{lemma}
\begin{proof} We observe first that if $k=1$ or $l=1$, we are trivially done, so we may assume that $k,l\geq 2$.

Suppose the statement is false. Let $A_1,\dots, A_k$ and $B_1,\dots, B_l$ be a configuration as described above, for which there exists no such $M$, and $|\cap_{i=1}^kA_i|-|\cup_{i=1}^lB_i|$ is minimal among such configurations. 

We begin by observing that if $A_i \subseteq A_j$ for some $i\neq j$, then we can simply omit $A_j$. Therefore, without loss of generality, we may assume the $A_i$'s are pairwise incomparable, and similarly, that the $B_i$'s are pairwise incomparable.
	
By assumption, there is no $M \in \mathcal F$ such that $\cup_{i=1}^l B_i \subseteq M \subseteq \cap_{i=1}^k A_i$. In particular, this means that $\cup_{i=1}^l B_i, \cap_{i=1}^k A_i \notin \mathcal F$. Therefore, each must form an $\mathcal N$ with three other elements of $\mathcal F$.
		
By Lemma~\ref{lem:positions of intersections/unions in N}, $\cup_{i=1}^lB_i \neq \cap_{i=1}^k A_i$ since $\cap_{i=1}^kA_i$ can only be a minimal element of such an $\mathcal N$, and $\cup_{i=1}^lB_i$ can only be a maximal element of such an $\mathcal N$. Thus $\cup_{i=1}^lB_i \subset \cap_{i=1}^k A_i$. Moreover, putting everything together, Lemma~\ref{lem:positions of intersections/unions in N} gives us the following diagram for some $X,Y,Z,X',Y',Z'\in\mathcal F$.
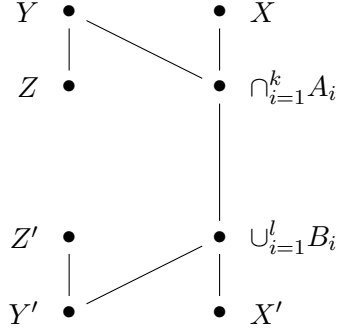
\begin{center}
\begin{tikzpicture}[scale=1]
\node at (0,2) (Y) {\textbf{$\bullet$}};
\node at (2,2) (X) {\textbf{$\bullet$}};
\node at (0,1) (Z) {\textbf{$\bullet$}};
\node at (2,1) (N) {\textbf{$\bullet$}};
\node at (2,-1) (U) {\textbf{$\bullet$}};
\node at (0,-2) (Y') {\textbf{$\bullet$}};
\node at (2,-2) (X') {\textbf{$\bullet$}};
\node at (0,-1) (Z') {\textbf{$\bullet$}};
\node[right=1pt of X] {$X$};
\node[left=1pt of Y] {$Y$};
\node[left=1pt of Z] {$Z$};
\node[right=1pt of X'] {$X'$};
\node[left=1pt of Y'] {$Y'$};
\node[left=1pt of Z'] {$Z'$};
\node[right=1pt of N] {$\cap_{i=1}^k A_i$};
\node[right=1pt of U] {$\cup_{i=1}^l B_i$};
\draw (Z) -- (Y);
\draw (X) -- (N);
\draw (N) -- (Y);
\draw (X') -- (U);
\draw (Y') -- (U);
\draw (Z') -- (Y');
\draw (U) -- (N);
\end{tikzpicture}
\end{center}
First, we observe that all sets in the diagram are pairwise distinct. Indeed $\cap_{i=1}^l A_i, \cup_{i=1}^l B_i \notin \mathcal F$ are distinct from $X, X', Y, Y', Z, Z' \in \mathcal F$. Moreover, $\cup_{i=1}^l B_i\subset\cap_{i=1}^k A_i \subset X, Y$, thus $X$ and $Y$ are distinct from $X', Y', Z'$. Similarly $X'$ and $Y'$ are distinct from $X, Y, Z$. We are left to show that $Z\neq Z'$. However, if $Z = Z'$, then $X,Y, Z, X'$ would form an $\mathcal N$ in $\mathcal F$, a contradiction.
		
Next, since by Lemma~\ref{lem:positions of intersections/unions in N} $\cap_{i=1}^k A_i$ cannot be the maximal element of an $\mathcal N$, we get that $Z',Y',X'$ and $\cap_{i=1}^k A_i$ cannot form an $\mathcal N$, thus $Z'\subset \cap_{i=1}^k A_i$. Similarly, we also must have $\cup_{i=1}^l B_i\subset Z$. We further observe that $Z$ cannot be a subset of $Z'$ as $Z'$ is a subset of $\cap_{i=1}^k A_i$, and if $Z$ and $Z'$ were incomparable, then $X,Y,Z,Z'$ would form an $\mathcal N$ in $\mathcal F$, a contradiction. Therefore we must have $Z'\subset Z$, and consequently $Z'\bigcup(\cup_{i=1}^l B_i)\subseteq Z\bigcap(\cap_{i=1}^kA_i$), as depicted in the diagram below.
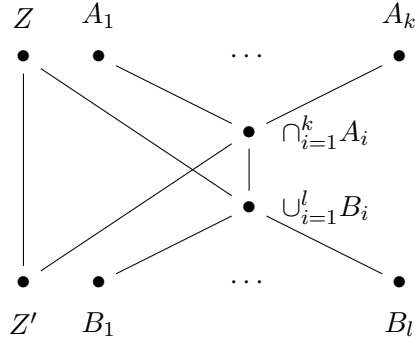
\begin{center}
\begin{tikzpicture}[scale=1]
\node at (0,-1) (B1) {\textbf{$\bullet$}};
\node at (0,2) (A1) {\textbf{$\bullet$}};
\node at (4,2) (Ak) {\textbf{$\bullet$}};
\node at (4,-1) (Bl) {\textbf{$\bullet$}};
\node at (2,-1) {$\cdots$};
\node at (2,2) {$\cdots$};
\node at (-1,2) (Z) {\textbf{$\bullet$}};
\node at (2,1) (N) {\textbf{$\bullet$}};
\node at (2,0) (U) {\textbf{$\bullet$}};
\node at (-1,-1) (Z') {\textbf{$\bullet$}};
\node[above=1pt of A1] {$A_1$};
\node[above=1pt of Ak] {$A_k$};
\node[below=1pt of B1] {$B_1$};
\node[below=1pt of Bl] {$B_l$};
\node[above=1pt of Z] {$Z$};
\node[below=1pt of Z'] {$Z'$};
\node[right=2pt of N] {$\cap_{i=1}^k A_i$};
\node[right=2pt of U] {$\cup_{i=1}^l B_i$};
\draw (U) -- (Bl);
\draw (U) -- (B1);
\draw (N) -- (Ak);
\draw (N) -- (A1);
\draw (N) -- (Z');
\draw (Z) -- (U);
\draw (Z') -- (Z);
\draw (U) -- (N);
\end{tikzpicture}
\end{center}
Using the fact that $Z$ and $\cap_{i=1}^kA_i$ are incomparable, and that $Z'$ and $\cup_{i=1}^l B_i$ are incomparable, we get that $\cup_{i=1}^{l} B_i \subset Z'\bigcup(\cup_{i=1}^l B_i )\subseteq Z \bigcap(\cap_{i=1}^k A_i )\subset \cap_{i=1}^k A_i$.  Therefore $(Z\bigcap(\cap_{i=1}^k A_i))\setminus (Z'\bigcup(\cup_{i=1}^l B_i))\subset \cap_{i=1}^k A_i\setminus\cup_{i=1}^l B_i$.
		
By the minimality of $|\cap_{i=1}^k A_i|-|\cup_{i=1}^l B_i|$, the claim must hold for the configuration $Z, A_1, \dots, A_k$, $Z', B_1, \dots, B_l$. Thus, there exists $ M \in \mathcal F$ such that $Z'\bigcup(\cup_{i=1}^l B_i)\subseteq M \subseteq Z\bigcap(\cap_{i=1}^k A_i)$. This however implies that $\cup_{i=1}^l B_i \subseteq M \subseteq \cap_{i=1}^k A_i$, a contradiction, finishing the proof of the lemma.
\end{proof}
The above lemma will be used in the paper in the case where $k=l=2$, or in other words, whenever we have butterflies in our $\mathcal N$-saturated family.
\begin{corollary}\label{cor:butterfly} Let $\mathcal F$ be an $\mathcal N$-saturated family, and $A, B, C, D \in \mathcal F$ such that they form a butterfly with maximal elements $A,B$. Then there exists some $M \in \mathcal F$ such that $C\cup D \subseteq M \subseteq A\cap B$.
\end{corollary}
	
Let $\mathcal G$ be a component of $\mathcal F$. Lemma~\ref{lem:max min in same comp are comparable} and \ref{lem:complete bip graph has a midpoint} tell us that $\mathcal G$ has a `midpoint' between its maximal elements and its minimal elements. It turns out that at least one of these `midpoints' is a `one-way valve' in the sense that everything in the component is comparable to it.

\begin{lemma}\label{lem:max/min midpoint comparable to everything} Let $\mathcal G$ be a component of $\mathcal F$. Let $B_1,\dots, B_l$ be the minimal elements, and $A_1,\dots,A_k$ be the maximal elements of $\mathcal G$. Let $M\in\mathcal F$ be such that $\cup_{i=1}^l B_i\subseteq M\subseteq\cap_{i=1}^k A_i$ and $M$ minimal with this property. Then every $X\in \mathcal G$ is comparable to $M$.
\end{lemma}
\begin{proof}
Suppose there exists $X\in \mathcal G$  such that $M$ and $X$ are incomparable. Since $\cup_{i=1}^l B_i\subseteq M$, then either $\cup_{i=1}^l B_i \subseteq X$, or $\cup_{i=1}^{l}B_i$ and $X$ are incomparable. In the former case, by Lemma~\ref{lem:complete bip graph has a midpoint}, there exists $M'$ such that $B_1, \dots, B_l$ are subsets of $M'$, and $M'$ is a subset of both $X$ and $M$, contradicting the minimality of $M$. In the latter case, there exists $B_j$ such that $B_j$ and $X$ are incomparable for some $1\leq j\leq l$. Furthermore, since $X \in \mathcal G$, there exists a minimal element comparable to it. Without loss of generality, let $B_1 \subset X$, which also implies that $j\neq1$. Then $B_1, X, M, B_j$ form an $\cN$, a contradiction.
\end{proof}
By symmetry, we of course have the following.
\begin{lemma}\label{lem:max/min midpoint comparable to everything2} Let $\mathcal G$ be a component of $\mathcal F$. Let $B_1,\dots, B_l$ be the minimal elements, and $A_1,\dots,A_k$ be the maximal elements of $\mathcal G$. Let $M\in\mathcal F$ be such that $\cup_{i=1}^l B_i\subseteq M\subseteq\cap_{i=1}^k A_i$ and $M$ maximal with this property. Then every $X\in \mathcal G$ is comparable to $M$.
\end{lemma}
	
\section{The setup and preliminary lemmas}
Let $\mathcal F$ be an $\mathcal N$-saturated family with ground set $[n]$, and $\mathcal G$ a component of $\mathcal F$. Let $M\in\mathcal F$ be a set less than all maximal elements and greater than all minimal elements of $\mathcal G$, which exists by Lemma~\ref{lem:complete bip graph has a midpoint}, and is minimal with this property. By Lemma~\ref{lem:max/min midpoint comparable to everything} every element of $\mathcal G$ is comparable to $M$. In this section we explore some properties of $\mathcal F$ that will be heavily used later in the paper.
\begin{lemma}\label{lem:simplified case analysis} Suppose $A'\notin\mathcal F$ is a set for which there exists $A\in\mathcal F$ such that $A'\subset A$, and $A'$ is a maximal element of an $\mathcal N$ in $\mathcal F\cup\{A'\}$. Then we may assume that the other maximal element is also a subset of $A$.
\end{lemma}
\begin{proof} Since $A' \notin \cF$ is a maximal element of some $\mathcal N$ in $\mathcal F\cup\{A'\}$, then we either have figure (i), or figure (ii) depicted below, for some $X,Y,Z,X',Y',Z'\in\mathcal F$.
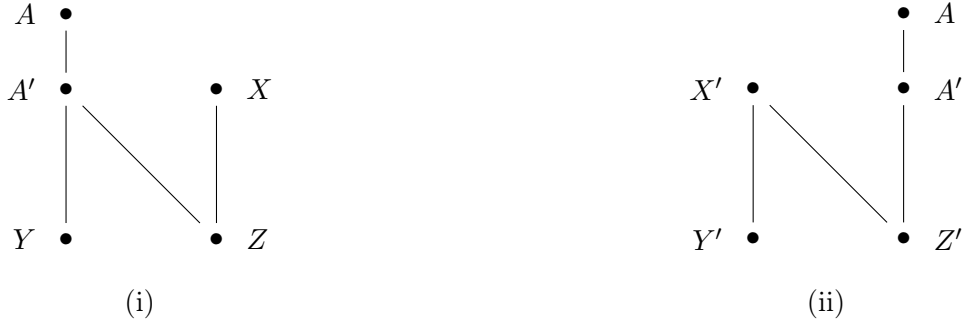
\begin{figure}[H]
\begin{minipage}{0.45\textwidth}
\centering
\begin{tikzpicture}[scale=1]
\node at (0,0) (Si) {\textbf{$\bullet$}};
\node at (0,2) (A) {\textbf{$\bullet$}};
\node at (2,2) (C) {\textbf{$\bullet$}};
\node at (2,0) (B) {\textbf{$\bullet$}};
\node at (0,3) (S) {\textbf{$\bullet$}};
\node[left=1pt of A] {$A'$};
\node[left=1pt of Si] {$Y$};
\node[right=1pt of C] {$X$};
\node[right=1pt of B] {$Z$};
\node[left=1pt of S] {$A$};
\draw (A) -- (B);
\draw (B) -- (C);
\draw (A) -- (Si);
\draw (S) -- (A);
\end{tikzpicture}
\caption*{(i)}
\end{minipage}
\hfill
\begin{minipage}{0.45\textwidth}
\centering
\begin{tikzpicture}[scale=1]
\node at (0,0) (Si) {\textbf{$\bullet$}};
\node at (0,2) (A) {\textbf{$\bullet$}};
\node at (2,2) (C) {\textbf{$\bullet$}};
\node at (2,0) (B) {\textbf{$\bullet$}};
\node at (2,3) (S) {\textbf{$\bullet$}};
\node[left=1pt of A] {$X'$};
\node[left=1pt of Si] {$Y'$};
\node[right=1pt of C] {$A'$};
\node[right=1pt of B] {$Z'$};
\node[right=1pt of S] {$A$};
\draw (A) -- (B);
\draw (B) -- (C);
\draw (A) -- (Si);
\draw (S) -- (C);
\end{tikzpicture}
\caption*{(ii)}
\end{minipage}
\end{figure}

In the first case (figure (i)), $A,X,Y, Z$ cannot form an $\mathcal N$, as they are all elements of $\mathcal F$. Thus $A$ and $X$ must be comparable. Since $A'\subset A$ and $A'$ and $X$ are incomparable, we must have $X\subset A$, as desired.
		
In the second case (figure (ii)), in order for $A,X',Y',Z'$ to not form an $\mathcal N$ in $\mathcal F$, $A$ must be comparable to $X'$ or $Y'$. If $X'$ and $A$ are comparable, as before, we must have $X'\subset A$. If $X'$ and $A$ are incomparable, then $Y'$ and $A$ are comparable, and consequently $Y'\subset A$.  This produces a butterfly in $\mathcal F$, namely $X',Y',Z',A$. By Corollary~\ref{cor:butterfly}, there exists $M'\in\mathcal F$ such that $Y',Z'\subset M'\subset X',A$. We observe now that $M'$ is incomparable to $A'$. Indeed, if $M' \subseteq A'$, then $Y' \subset A'$, a contradiction, and if $A' \subset M'$, then $A' \subset X'$, a contradiction. Therefore $M',A',Y',Z'$ form an $\mathcal N$ in which the other maximal element is $M'$, a subset of $A$, as desired.
\end{proof}
By symmetry, we also have the following lemma.
\begin{lemma}\label{lem:simplified case analysis2} Suppose $B'\notin\mathcal F$ is a set for which there exists $B\in\mathcal F$ such that $B\subset B'$, and $B'$ is a minimal element of an $\mathcal N$ in $\mathcal F\cup\{B'\}$. Then we may assume that the other minimal element also contains $B$ as a subset.
\end{lemma}
\begin{lemma}\label{lem:when MU{i} is maximal in N} Let $M$ be as above. Suppose $i \notin M$ and $M \cup \{i\} \notin \mathcal F$. If $M\cup\{i\}$ is a maximal element of an $\mathcal N$ in $\mathcal F\cup\{M\cup\{i\}\}$, then we may assume that it is the maximal comparable to both minimal elements, and that one of the minimal elements is $M$. Moreover, there are at most $|\mathcal F|-2$ such singletons.
\end{lemma}
\begin{proof}
Since $M \cup \{i\}\notin\mathcal F$, then it must form an $\mathcal N$ when added to $\mathcal F$. If it is a maximal element of such an $\mathcal N$ then one of the two cases below holds.
		
Case 1. $M\cup\{i\}$ is comparable to both minimal elements of $\mathcal N$, as illustrated below, where $A_i, B_i, C_i\in\mathcal F$.
		
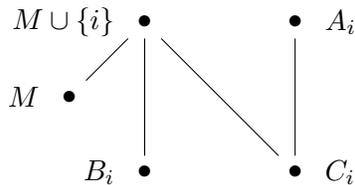
\begin{center}
\begin{tikzpicture}[scale=1]
\node at (0,2) (Si) {\textbf{$\bullet$}};
\node at (2,2) (A) {\textbf{$\bullet$}};
\node at (0,0) (B) {\textbf{$\bullet$}};
\node at (2,0) (C) {\textbf{$\bullet$}};
\node at (-1,1) (S) {\textbf{$\bullet$}};
\node[right=1pt of A] {$A_i$};
\node[left=1pt of B] {$B_i$};
\node[right=1pt of C] {$C_i$};
\node[left=1pt of Si] {$M \cup \{i\}$};
\node[left=1pt of S] {$M$};
\draw (A) -- (C);
\draw (B) -- (Si);
\draw (C) -- (Si);
\draw (S) -- (Si);
\end{tikzpicture}
\end{center}
If $C_i=M$, or $B_i=M$, then we are done. Therefore, we may assume that $B_i\neq M$ and $C_i\neq M$. Now, if both $A_i$ and $C_i$ are incomparable to $M$, then $M\cup \{i\}, M, A_i, C_i$ form an $\mathcal N$ where one of the minimal elements is $M$, as desired. Therefore, we may assume that at least one of $A_i$ or $C_i$ is comparable to $M$, which implies that $A_i, C_i, M$ are in the same component of $\mathcal F$. By Lemma~\ref{lem:max/min midpoint comparable to everything}, both $A_i, C_i$ are comparable to $M$. Since $A_i$ and $M\cup\{i\}$ are incomparable, we must have $M\subset A_i$. Also, we cannot have $M\subseteq C_i$ as $C_i\subset M\cup\{i\}$ and $C_i\neq M$. Thus $C_i \subset M \subset A_i$.

Finally, we observe that $B_i$ and $M$ are incomparable. Indeed, by cardinality we cannot have $M\subset B_i$, and if $B_i\subset M$ then $B_i\subset A_i$, a contradiction. Therefore $M\cup \{i\}, M, A_i, B_i$ form the desired $\mathcal N$.
		
Case 2. $M\cup\{i\}$ is comparable to exactly one minimal element of $\mathcal N$, as illustrated below, where $A_i, B_i, C_i\in\mathcal F$.
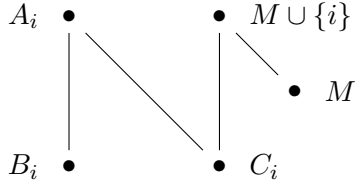
\begin{center}
\begin{tikzpicture}[scale=1]
\node at (0,0) (B) {\textbf{$\bullet$}};
\node at (0,2) (A) {\textbf{$\bullet$}};
\node at (2,2) (Si) {\textbf{$\bullet$}};
\node at (2,0) (C) {\textbf{$\bullet$}};
\node at (3,1) (S) {\textbf{$\bullet$}};
\node[left=1pt of A] {$A_i$};
\node[left=1pt of B] {$B_i$};
\node[right=1pt of C] {$C_i$};
\node[right=1pt of Si] {$M \cup \{i\}$};
\node[right=1pt of S] {$M$};
\draw (A) -- (B);
\draw (Si) -- (C);
\draw (A) -- (C);
\draw (S) -- (Si);
\end{tikzpicture}
\end{center}
If $C_i=M$, then $A_i, B_i, M$ are in the same component of $\mathcal F$. However, $B_i$ is incomparable to $C_i = M$, which contradicts Lemma~\ref{lem:max/min midpoint comparable to everything}. Therefore $C_i \neq M$. 
		
If both $A_i$ and $C_i$ are incomparable to $M$, then $M\cup \{i\}, M, A_i, C_i$ form the desired $\mathcal N$, where $M\cup\{i\}$ is the maximal comparable to both minimal elements, and $M$ is one of the minimal elements. Therefore, we may assume that at least one of $A_i$ or $C_i$ is comparable to $M$. This now means that $A_i, B_i, C_i, M$ are in the same component of $\mathcal F$, thus, by Lemma~\ref{lem:max/min midpoint comparable to everything}, $A_i, B_i, C_i$ are all comparable to $M$.  Since $B_i$ is incomparable to $M \cup \{i\}$, we must have $M \subset B_i$. Moreover, since $C_i \subset M \cup \{i\}$ and $C_i\neq M$, we must have $C_i\subset M$. But this implies that $C_i \subset M \subset B_i$, a contradiction.
		
For the second part of the lemma, pick $X_i$ to be the other minimal element (distinct from $M$) of the $\mathcal N$ in $\mathcal F\cup\{ M\cup\{i\}\}$ established earlier. We therefore have $X_i\subset M\cup\{i\}$, and $M$ and $X_i$ are incomparable, which implies that $i\in X_i$. Thus, $X_i\neq X_j$ if $i \neq j$, as otherwise $i \in X_i = X_j$, which implies that $i\in M\cup \{j\}$, and so $i\in M$, a contradiction. Moreover, $X_i\notin\{\emptyset, [n]\}$ as it is incomparable to $M$. Putting everything together we get that there are at most $|\mathcal F|-2$ singletons $i$ such that $i\notin M$, and $M \cup \{i\}\notin \mathcal F$ is a maximal element of an $\mathcal N$ in $\mathcal F\cup\{M\cup\{i\}\}$.
\end{proof}
The following lemma allows us to assume that the set $M$ has relatively small cardinality.
\begin{lemma}\label{lem:wlog F satisfies certain conditions} There exists an $\mathcal N$-saturated family $\mathcal F'$ with ground set $[n]$ such that $|\mathcal F| = |\mathcal F'|$, and at least one of its components $\mathcal G'$ has a set of size at most $n/2$ that is comparable to all elements of the component.
\end{lemma}
\begin{proof}
Let $\mathcal G$ be a component of $\mathcal F$. If the minimal set $M$ that is between the maximal and the minimal elements of $\mathcal G$ satisfies $|M|\leq n/2$, we are done. Otherwise, let $\mathcal F' = \{[n]\setminus X: X \in \mathcal F\}$. Since $\mathcal N$ is invariant under taking complements, $\mathcal F'$ is an $\mathcal N$-saturated family with $\mathcal G' = \{[n]\setminus X: X \in \mathcal G\}$ as one of its components. Moreover, $|\mathcal F|=|\mathcal F'|$. Let $M'$ be a minimal set comparable to all sets of $\mathcal G'$, which exists by Lemma~\ref{lem:max/min midpoint comparable to everything}. Therefore $M'\subseteq [n]\setminus M$, thus $|M'| \leq |[n]\setminus M|<n/2$, as desired.
\end{proof}
\section{Proof of the main result}
We are now ready to prove that the saturation number of $\mathcal N$ is at least linear. The main part of the proof is in fact Proposition~\ref{prop:main result}, which easily implies the linear lower bound in Theorem~\ref{thm:lower bound on saturation number}.
\begin{proposition}\label{prop:main result} Let $\mathcal F$ be an $\mathcal N$-saturated family with ground set $[n]$. Let $\mathcal G$ be a component of $\mathcal F$ such that there exists $M\in\mathcal F$ minimal with the property that it is comparable to all sets of $\mathcal G$, and $|M|\leq n/2$. Then, for at least $n/2 -|\mathcal F| + 2$ singletons $i \in [n]$, there exists $S_i$ such that $i\notin S_i$ and $S_i,S_i \cup \{i\}\in \mathcal F$.
\end{proposition}
\begin{proof}
Since $|M|\leq n/2$, there exist at least $n/2$ singletons $i$ such that $i \notin M$. By Lemma~\ref{lem:when MU{i} is maximal in N}, there are at least $n/2-(|\mathcal F| -2)$ of singletons $i$ such that $i\notin M$, and either $M \cup \{i\} \in \mathcal F$, or $M \cup \{i\} \notin \mathcal F$ and it can only be a minimal element of an $\mathcal N$ in $\mathcal F\cup\{M\cup\{i\}\}$. We will show that for every such $i$, there exists $S_i$ such that $i\notin S_i$ and $S_i, S_i \cup \{i\}\in \mathcal F$, which will finish the proof.
		
Suppose for a contradiction that there exists such a singleton $i$ for which we cannot find two sets in $\mathcal F$ that differ exactly by $i$. This implies that $M \cup \{i\} \notin \mathcal F$, and by assumption, it can only be a minimal element of an $\mathcal N$ in $\mathcal F\cup\{M\cup\{i\}\}$, as illustrated below, where $A$ is the maximal element comparable to both minimal elements, and $B$ is the other minimal element of the $\mathcal N$.
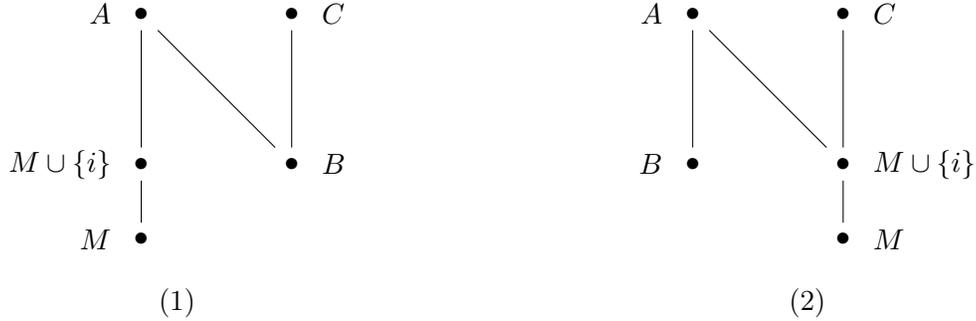
\begin{figure}[H]
\begin{minipage}{0.50\textwidth}
\centering
\begin{tikzpicture}[scale=1]
\node at (0,0) (Si) {\textbf{$\bullet$}};
\node at (0,2) (A) {\textbf{$\bullet$}};
\node at (2,2) (C) {\textbf{$\bullet$}};
\node at (2,0) (B) {\textbf{$\bullet$}};
\node at (0,-1) (S) {\textbf{$\bullet$}};
\node[left=1pt of A] {$A$};
\node[left=1pt of Si] {$M \cup \{i\}$};
\node[right=1pt of C] {$C$};
\node[right=1pt of B] {$B$};
\node[left=1pt of S] {$M$};
\draw (A) -- (B);
\draw (B) -- (C);
\draw (A) -- (Si);
\draw (S) -- (Si);
\end{tikzpicture}
\caption*{(1)}
\end{minipage}
\hfill
\begin{minipage}{0.50\textwidth}
\centering
\begin{tikzpicture}[scale=1]
\node at (0,0) (B) {\textbf{$\bullet$}};
\node at (0,2) (A) {\textbf{$\bullet$}};
\node at (2,2) (C) {\textbf{$\bullet$}};
\node at (2,0) (Si) {\textbf{$\bullet$}};
\node at (2,-1) (S) {\textbf{$\bullet$}};
\node[left=1pt of A] {$A$};
\node[right=1pt of Si] {$M \cup \{i\}$};
\node[right=1pt of C] {$C$};
\node[left=1pt of B] {$B$};
\node[right=1pt of S] {$M$};
\draw (A) -- (B);
\draw (C) -- (Si);
\draw (A) -- (Si);
\draw (S) -- (Si);
\end{tikzpicture}
\caption*{(2)}
\end{minipage}
\end{figure}

Among all the possible $\mathcal N$ formed with $M \cup \{i\}$, we pick the one having $|B|$ maximal, and then under this choice of $B$, choose $|A|$ minimal. Since $M\subset M\cup\{i\}$, we have that $A, B, C$ are in the same component of $\mathcal F$ as $M$, therefore they are all comparable to $M$ by Lemma~\ref{lem:max/min midpoint comparable to everything}. This tells us that $M\subset B$. Since $M\cup\{i\}$ and $B$ are incomparable, we must have that $i\notin B$. Moreover, in case (1), $B\subset C$, and since $M\cup\{i\}$ and $C$ are incomparable, we also get that $i\notin C$. By our assumption, $B\cup\{i\}\notin\mathcal F$, thus it must form an $\mathcal N$ when added to $\mathcal F$.	
\begin{claim} $B \cup \{i\}$ cannot be a minimal element of any $\mathcal N$ in $\mathcal F\cup\{B\cup\{i\}\}$.
\end{claim}
\begin{proof} Suppose that $B\cup\{i\}$ is one of the minimal elements of an $\mathcal N$. Then one of the two cases depicted below holds, where $X,Y,Z, X',Y',Z'\in\mathcal F$.
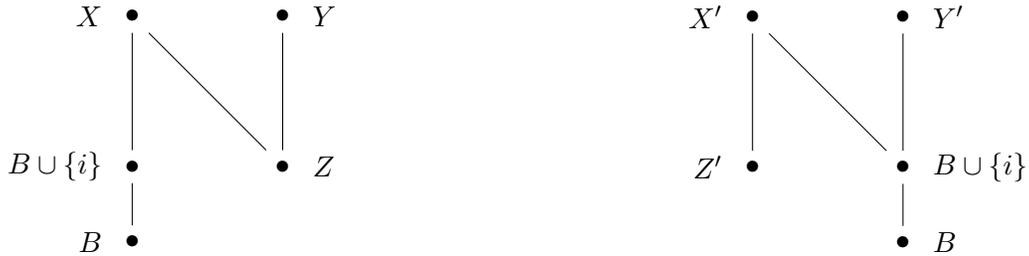
\begin{figure}[H]
\begin{minipage}{0.45\textwidth}
\centering
\begin{tikzpicture}[scale=1]
\node at (0,0) (Si) {\textbf{$\bullet$}};
\node at (0,2) (A) {\textbf{$\bullet$}};
\node at (2,2) (C) {\textbf{$\bullet$}};
\node at (2,0) (B) {\textbf{$\bullet$}};
\node at (0,-1) (S) {\textbf{$\bullet$}};
\node[left=1pt of A] {$X$};
\node[left=1pt of Si] {$B \cup \{i\}$};
\node[right=1pt of C] {$Y$};
\node[right=1pt of B] {$Z$};
\node[left=1pt of S] {$B$};
\draw (A) -- (B);
\draw (B) -- (C);
\draw (A) -- (Si);
\draw (S) -- (Si);
\end{tikzpicture}
\end{minipage}
\hfill
\begin{minipage}{0.45\textwidth}
\centering
\begin{tikzpicture}[scale=1]
\node at (0,0) (B) {\textbf{$\bullet$}};
\node at (0,2) (A) {\textbf{$\bullet$}};
\node at (2,2) (C) {\textbf{$\bullet$}};
\node at (2,0) (Si) {\textbf{$\bullet$}};
\node at (2,-1) (S) {\textbf{$\bullet$}};
\node[left=1pt of A] {$X'$};
\node[right=1pt of Si] {$B \cup \{i\}$};
\node[right=1pt of C] {$Y'$};
\node[left=1pt of B] {$Z'$};
\node[right=1pt of S] {$B$};
\draw (A) -- (B);
\draw (C) -- (Si);
\draw (A) -- (Si);
\draw (S) -- (Si);
\end{tikzpicture}
\end{minipage}
\end{figure}

By Lemma~\ref{lem:simplified case analysis2} we may assume that $B \subset Z$ in the first case, and $B\subset Z'$ in the second case. Since $B \cup \{i\}$ incomparable to $Z$ (or $Z'$), we must have that $i \notin Z$ (or $i\notin Z')$. Since $M \subset B \subset Z$ (or $M\subset B\subset Z'$), we get that $M \cup \{i\}$ is incomparable to $Z$ (or $Z'$). This now implies that $M \cup \{i\}, X, Y, Z$ (or $M\cup\{i\}, X', Y', Z')$ form an $\mathcal N$ with $M \cup \{i\}, Z$ (or $Z'$) being the minimal elements, contradicting the maximality of $|B|$, replaced by $Z$ (or $Z'$).
\end{proof}
\begin{claim} There exists $D \in \mathcal F$ such that $M \cup \{i\} \subset D \subset B \cup \{i\}$.
\end{claim}
\begin{proof}
By Claim A, $B \cup \{i\}$ is always a maximal element of any $\mathcal N$ in $\mathcal F\cup\{B\cup\{i\}\}$. We therefore have two cases, depicted below, for some $X,Y,Z,X',Y',Z'\in\mathcal F$.
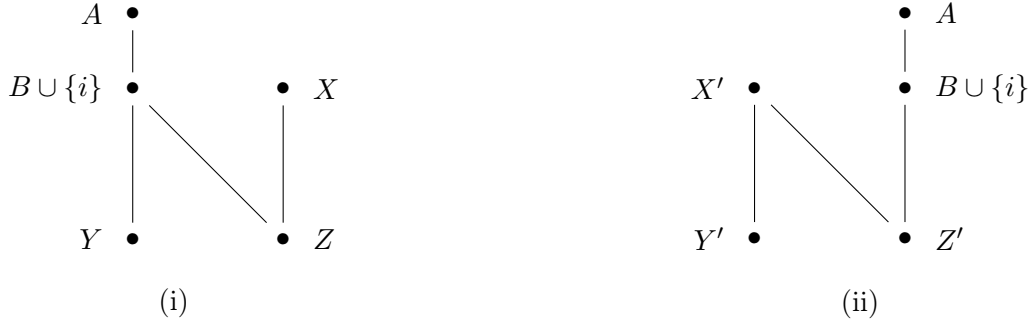
\begin{figure}[H]
\begin{minipage}{0.45\textwidth}
\centering
\begin{tikzpicture}[scale=1]
\node at (0,2) (Si) {\textbf{$\bullet$}};
\node at (2,2) (A) {\textbf{$\bullet$}};
\node at (0,0) (B) {\textbf{$\bullet$}};
\node at (2,0) (C) {\textbf{$\bullet$}};
\node at (0, 3) (AA)  {\textbf{$\bullet$}};
\node[right=1pt of A] {$X$};
\node[left=1pt of B] {$Y$};
\node[right=1pt of C] {$Z$};
\node[left=1pt of Si] {$B \cup \{i\}$};
\node[left=1pt of AA] {$A$};
\draw (A) -- (C);			
\draw (B) -- (Si);
\draw (C) -- (Si);
\draw (AA) -- (Si);
\end{tikzpicture}
\caption*{(i)}
\end{minipage}
\hfill
\begin{minipage}{0.45\textwidth}
\centering
\begin{tikzpicture}[scale=1]
\node at (0,0) (B) {\textbf{$\bullet$}};
\node at (0,2) (A) {\textbf{$\bullet$}};
\node at (2,2) (Si) {\textbf{$\bullet$}};
\node at (2,0) (C) {\textbf{$\bullet$}};
\node at (2,3) (AA) {\textbf{$\bullet$}};
\node[left=1pt of A] {$X'$};
\node[left=1pt of B] {$Y'$};
\node[right=1pt of C] {$Z'$};
\node[right=1pt of Si] {$B \cup \{i\}$};
\node[right=1pt of AA] {$A$};
\draw (A) -- (B);
\draw (Si) -- (C);
\draw (Si) -- (AA);
\draw (A) -- (C);
\end{tikzpicture}
\caption*{(ii)}
\end{minipage}
\end{figure}
			
In case (i), if $i \notin Y\cup Z$, then $Y, Z \subseteq B$, and $B \not \subseteq X$ as $Y$ is incomparable to $X$, so $Z\neq B$. Moreover, since $X$ is incomparable to $B \cup \{i\}$, $X \not \subseteq B$. Therefore $X$ and $B$ are incomparable. If $Y\neq B$, then $B, X, Y, Z$ form an $\mathcal N$ in $\mathcal F$, a contradiction. If $Y=B$, then $Z$ is incomparable to $B$ but $Z \subset B\cup \{i\}$, so $i\in Z$, a contradiction. Hence $i \in Y$ or $i\in Z$, and let $D=Y$ or $D= Z$ such that $i\in D$.

Since $Y, Z \subset A$ and $M\subset A$, all of $X, Y, Z, A, M$ are in the same component $\mathcal G$ of $\mathcal F$, thus comparable to $M$ by Lemma~\ref{lem:max/min midpoint comparable to everything}. Since $i\notin M$, we can only have $M\subset D$. This immediately implies that $M\cup\{i\}\subset D\subset B\cup\{i\}$ as desired.
			
In case (ii), we begin by showing that $i\in Z'$. Suppose that $i\notin Z'$, then $Z' \subseteq B$. We also observe that $Y'$ is incomparable to $B$, since $Y'$ is incomparable to both $Z'$ and $B \cup \{i\}$. Next, $X', Y', Z', B$ cannot form an $\mathcal N$ in $\mathcal F$, so we must have $B \subset X'$ or $B=Z'$. In either case we get $M \subset B\subset X'$. Since $X'$ is incomparable to $B \cup \{i\}$, we get that $i \notin X'$, and consequently $X'$ is incomparable to $M \cup \{i\}$. Furthermore, by Lemma~\ref{lem:simplified case analysis}, we may assume that $X' \subset A$. 

Putting everything together we get that $X', Y', Z'$ are all in the component $\mathcal G$ of $\mathcal F$, and hence comparable to $M$. In particular, this means that $M \subset Y'$, as otherwise $Y' \subseteq M \subset B \cup \{i\}$, a contradiction.
			
We recall that $M \cup \{i\}$ is a minimal element of an $\mathcal N$ as depicted in the first picture of this proof. We consider the cases (1) and (2) separately.
			
In the first case corresponding to figure (1), we see that $Y'\subset X' \subset A$. Since $A, B, C, Y'$ cannot form an $\cN$, we must have $Y' \subset C$. But now $A, B, C, Y'$ form a butterfly, and so, by Corollary~\ref{cor:butterfly}, there exists $M'\in\mathcal F$ such that $Y',B\subset M'\subset A,C$. We observe that $M\subset C$ (to avoid $A,M,B,C$ from forming an $\mathcal N$), thus $i\notin C$. Since $M' \subset C$, we also get that $i \notin M'$. Moreover, since $M \subset Y' \subset M'$, we get that $M'$ is incomparable to $M \cup \{i\} $. Finally, $M\cup \{i\}, A, M', C$ form an $\mathcal N$, and $B \subset M'$, contradicting the maximality of $|B|$ (replaced by $M'$).
			
In the second case corresponding to figure (2), we first note that $X'$ is incomparable to $C$. Indeed, if $X' \subset C$ then $B \subset C$, a contradiction, and since $i \in C$, we cannot have $C\subseteq X'$. Therefore $M\cup \{i\}, A, X', C$ form an $\mathcal N$, which contradicts the maximality of $|B|$ as $B\subset X'$.
			
Therefore we must have that $i \in Z'$. Let $D= Z'\subset B\cup\{i\}$. Since $D \subset A$, we get that $D \in \mathcal G$, thus $D$ is comparable to $M$. This gives $M\subset D$ since $i\in D$. Finally, we get that $M \cup \{i\} \subset D \subset B \cup \{i\}$ as claimed.
\end{proof}
We now observe now that Claim B implies that case (1) cannot happen, i.e. $M \cup \{i\}$ cannot be the minimal element of an $\mathcal N$ that is comparable to exactly one maximal (as depicted in (1)). Indeed, if that were the case, we have that $D\subset B\cup\{i\}\subset A$, and $C$ and $D$ are incomparable (as $i \notin C$, $i \in D$, and $C \not \subset D$ otherwise $C \subset A$). This now implies that $A, B, C, D$ form an $\mathcal N$ in $\mathcal F$, a contradiction.
		
Therefore, $M\cup\{i\}$ can only be the minimal element of an $\mathcal N$ that is comparable to both maximal elements. In that case, $B$ and $D$ are incomparable  ($i\in D$, so $D\not\subseteq B$, and $B\not\subset D$ since $D\subset B\cup\{i\}$), which in turn implies that $D$ and $C$ are incomparable (if $D\subset C$, then $A,B,D,C$ form an $\mathcal N$ in $\mathcal F$, and if $C\subseteq D$, then $C\subset A$). Therefore, we have the following Hasse diagram.
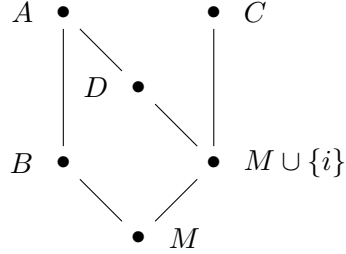
\begin{center}
\begin{tikzpicture}[scale=1]
\node at (0,0) (B) {\textbf{$\bullet$}};
\node at (0,2) (A) {\textbf{$\bullet$}};
\node at (2,2) (C) {\textbf{$\bullet$}};
\node at (1,1) (Z) {\textbf{$\bullet$}};					\node at (2,0) (Si) {\textbf{$\bullet$}};
\node at (1,-1) (S) {\textbf{$\bullet$}};
\node[left=1pt of A] {$A$};
\node[right=1pt of Si] {$M \cup \{i\}$};
\node[left=1pt of Z] {$D$};			
\node[right=1pt of C] {$C$};
\node[left=1pt of B] {$B$};
\node[right=1pt of S] {$M$};
\draw (A) -- (B);
\draw (S) -- (B);
\draw (Z) -- (Si);
\draw (A) -- (Z);
\draw (C) -- (Si);
\draw (S) -- (Si);
\end{tikzpicture}
\end{center}
The next claim tells us that, by our choice of $A$ and $B$, there is not much `room' between them.
\begin{claim} Let $T,T' \in \mathcal F$. If $B \subseteq T \subseteq A$, then $T = B$ or $T = A$. Also, if $B \subset T'$, then $A \subseteq T'$.
\end{claim}
\begin{proof}
Suppose that $T\neq A$ and $T\neq B$. Then $M\subset B \subset T \subset A$. Since $C$ is incomparable to both $A$ and $B$, it is therefore incomparable to $T$. If $i \notin T$, then $A, T, C, M\cup \{i\}$ form an $\mathcal N$, contradicting the maximality of $|B|$ (replace by $T$). If $i \in T$, then $T, B, C, M\cup \{i\}$ form an $\mathcal N$, contradicting the minimality of $|A|$ (replaced by $T$). Thus, $T=A$ or $T=B$.
			
Next, suppose that $B \subset T'$ but $A \not \subseteq T'$. By the first part of the claim, $T' \not \subset A$, hence $T'$ must be incomparable to $A$. For $T', B, A, D$ to not form an $\mathcal N$ in $\mathcal F$, we must have have $D \subset T'$. However, now $T',B,A,D$ form a butterfly in $\mathcal F$. By Corollary~\ref{cor:butterfly}, there exists $M'\in\mathcal F$ such that $B,D \subset M' \subset A,T'$, contradicting the first part of the claim.
\end{proof}
We now look at $A\setminus\{i\}$. By assumption, $A\setminus\{i\} \notin \mathcal F$, thus it must form an $\mathcal N$ when added to $\mathcal F$.
\begin{claim}$A\setminus\{i\}$ cannot be a minimal element of any $\mathcal N$ in $\mathcal F\cup\{A\setminus\{i\}\}$.
\end{claim}
\begin{proof}By assumption, $A\setminus\{i\} \notin \mathcal{F}$, thus it must form an $\mathcal N$ when added to $\mathcal F$. Suppose it is one of the minimal elements of such an $\mathcal N$. We therefore have the two cases illustrated below, where $X,Y,Z,X',Y',Z'\in\mathcal F$.
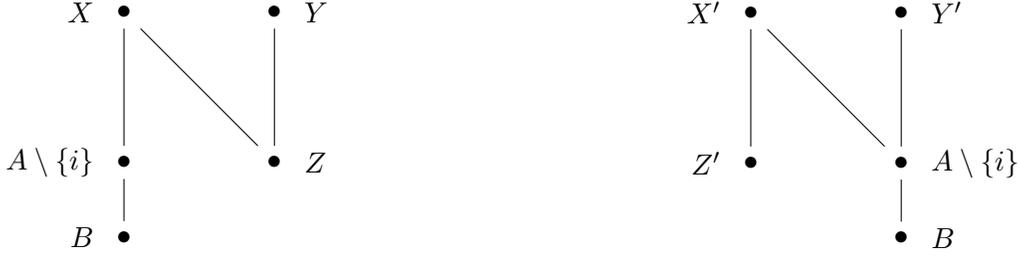
\begin{figure}[H]
\begin{minipage}{0.45\textwidth}
\centering
\begin{tikzpicture}[scale=1]
\node at (0,0) (Si) {\textbf{$\bullet$}};
\node at (0,2) (A) {\textbf{$\bullet$}};
\node at (2,2) (C) {\textbf{$\bullet$}};
\node at (2,0) (B) {\textbf{$\bullet$}};
\node at (0,-1) (S) {\textbf{$\bullet$}};
\node[left=1pt of A] {$X$};
\node[left=1pt of Si] {$A\setminus\{i\}$};
\node[right=1pt of C] {$Y$};
\node[right=1pt of B] {$Z$};
\node[left=1pt of S] {$B$};
\draw (A) -- (B);
\draw (B) -- (C);
\draw (A) -- (Si);
\draw (S) -- (Si);
\end{tikzpicture}
\end{minipage}
\hfill
\begin{minipage}{0.45\textwidth}
\centering
\begin{tikzpicture}[scale=1]
\node at (0,0) (B) {\textbf{$\bullet$}};
\node at (0,2) (A) {\textbf{$\bullet$}};
\node at (2,2) (C) {\textbf{$\bullet$}};
\node at (2,0) (Si) {\textbf{$\bullet$}};
\node at (2,-1) (S) {\textbf{$\bullet$}};
\node[left=1pt of A] {$X'$};
\node[right=1pt of Si] {$A\setminus\{i\}$};
\node[right=1pt of C] {$Y'$};
\node[left=1pt of B] {$Z'$};
\node[right=1pt of S] {$B$};
\draw (A) -- (B);
\draw (C) -- (Si);
\draw (A) -- (Si);
\draw (S) -- (Si);
\end{tikzpicture}
\end{minipage}
\end{figure}
By Lemma~\ref{lem:simplified case analysis2}, we may assume that $B \subset Z$ in the first case, and $B\subset Z'$ in the second case. Consequently, this gives by Claim C that $A \subseteq Z$ (or $Z'$ in the second case), thus $A\setminus\{i\} \subset Z$ (or $Z'$ in the second case), a contradiction.
\end{proof}
We are now going to inductively define sets $A_m, B_m\in\mathcal F$ for all integers $m\geq 0$ as follows. We set $A_0=A$ and $B_0=B$. We know that $A_0\setminus\{i\}\notin \mathcal F$ can only be a maximal element of any $\mathcal N$ in $\mathcal F\cup\{A_0\setminus\{i\}\}$. Among all such possible configurations, we define $B_1$ to be the minimal element comparable to both maximal elements of such an $\mathcal N$, of maximum cardinality. Having chosen $B_1$, we then define $A_1$ to be the other maximal element of such an $\mathcal N$ with minimal cardinality. This is illustrated below, where $K_1,K_1'\in\mathcal F$.
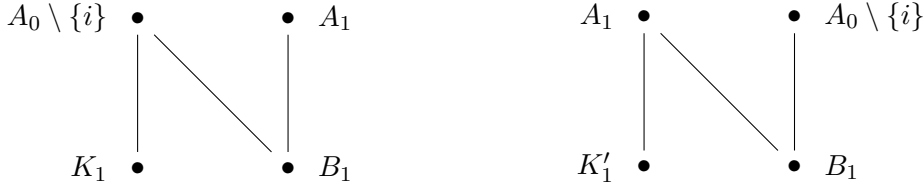
\begin{figure}[H]
\begin{minipage}{0.45\textwidth}
\centering
\begin{tikzpicture}[scale=1]
\node at (0,2) (Si) {\textbf{$\bullet$}};
\node at (2,2) (A) {\textbf{$\bullet$}};
\node at (0,0) (B) {\textbf{$\bullet$}};
\node at (2,0) (C) {\textbf{$\bullet$}};
\node[right=1pt of A] {$A_1$};
\node[left=1pt of B] {$K_1$};
\node[right=1pt of C] {$B_1$};
\node[left=1pt of Si] {$A_{0}\setminus\{i\}$};
\draw (A) -- (C);
\draw (B) -- (Si);
\draw (C) -- (Si);		
\end{tikzpicture}
\end{minipage}
\begin{minipage}{0.45\textwidth}
\centering
\begin{tikzpicture}[scale=1]
\node at (0,0) (B) {\textbf{$\bullet$}};
\node at (0,2) (A) {\textbf{$\bullet$}};
\node at (2,2) (Si) {\textbf{$\bullet$}};
\node at (2,0) (C) {\textbf{$\bullet$}};
\node[left=1pt of A] {$A_1$};
\node[left=1pt of B] {$K_1'$};
\node[right=1pt of C] {$B_1$};
\node[right=1pt of Si] {$A_0\setminus\{i\}$};
\draw (A) -- (B);
\draw (Si) -- (C);
\draw (C) -- (A);
\end{tikzpicture}
\end{minipage}
\end{figure}
We will show that $i\in A_1$ and that $A_1\setminus\{i\}\notin\mathcal F$ can only be the maximal element of an $\mathcal N$ in $\mathcal F\cup\{A_1\setminus\{i\}\}$. After that, we repeat the process, where we define $B_2\subset A_2$ from $B_1\subset A_1$, the same way we defined $B_1\subset A_1$ from $B_0\subset A_0$. In order to show that the process can always continue, we will show, by induction on $m$ the following properties.
		
\noindent\textbf{Claim 1.} \textit{$A_m \subset A_{m-1}$, and $A_m$ is incomparable to $B_{m-1}$. Moreover, we either have  $B_m$ and $B_{m-1}$ incomparable, or $B_m \subset B_{m-1}$.}
	
\noindent\textbf{Claim 2.} \textit{If $B_m \subset T$, where $T\in\mathcal F$, then $T$ is comparable to $A_{m-1}$.}
			
\noindent\textbf{Claim 3.} \textit{$B_m \cup \{i\}$ can only be a maximal element of an $\mathcal N$ in $\mathcal F\cup\{B_m\cup\{i\}\}$.}
			
\noindent\textbf{Claim 4.} \textit{There exists $D_m \in \mathcal F$ such that $M\cup \{i\}\subset D_m\subset B_m\cup \{i\}$.}
			
\noindent\textbf{Claim 5.} \textit{Let $T\in\mathcal F$. If $B_m \subseteq T \subseteq A_m$, then $T = A_m$ or $T = B_m$. If $B_m \subset T$, then $A_m \subseteq T$.}
			
\noindent\textbf{Claim 6.}  \textit{$A_m\setminus\{i\}$ can only be a maximal element in an $\mathcal N$ in $\mathcal F\cup\{A_m\setminus\{i\}\}$.}

Note that since $B_m\subset A_{m-1}\setminus\{i\}$, we have by construction that $i\notin B_m$ for all $m\geq 0$. Moreover, since $A_m\subset A_{m-1}$, but $A_m$ and $A_{m-1}\setminus\{i\}$ are incomparable, we must have $i\in A_m$ for all $m\geq 0$. These observations, together with Claim 4, give us the following picture.
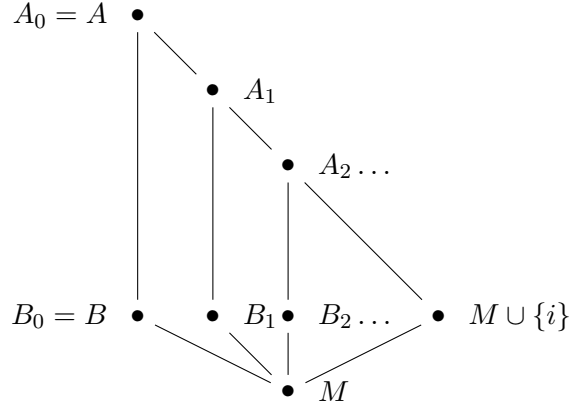
\begin{center}
\begin{tikzpicture}[scale=1]
\node at (0,0) (B) {\textbf{$\bullet$}};
\node at (0,4) (A) {\textbf{$\bullet$}};
\node at (1,3) (A1) {\textbf{$\bullet$}};
\node at (1,0) (B1) {\textbf{$\bullet$}};
\node at (2,2) (A2) {\textbf{$\bullet$}};
\node at (2,0) (B2) {\textbf{$\bullet$}};
\node at (4,0) (Mu) {\textbf{$\bullet$}};
\node at (2,-1) (M) {\textbf{$\bullet$}};
\node[left=1pt of A] {$A_0=A$};
\node[right=1pt of A1] {$A_1$};
\node[right=1pt of A2] {$A_2 \dots$};
\node[left=1pt of B] {$B_0=B$};
\node[right=1pt of B1] {$B_1$};
\node[right=1pt of B2] {$B_2 \dots$};
\node[right=1pt of Mu] {$M \cup \{i\}$};
\node[right=1pt of M] {$M$};				
\draw (A) -- (B);
\draw (A) -- (A1);
\draw (A2) -- (A1);
\draw (A2) -- (B2);
\draw (B) -- (M);
\draw (B1) -- (M);
\draw (B2) -- (M);
\draw (Mu) -- (M);												
\draw (A2) -- (Mu);
\draw (A1) -- (B1);
\end{tikzpicture}
\end{center}
Once these claims have been established, a contradiction is immediately reached as the process must terminate between $A_0$ and $M\cup\{i\}$, hence in  finite number of steps.

We prove these claims by induction on $m$. When $m=0$, Claims 1 and 2 are vacuously true, and Claims 3, 4, 5 and 6 are Claims A, B, C and D from above. Let now $m\geq 1$, and assume that they are true for all $m'<m$.
\begin{claim1} $A_m \subset A_{m-1}$, and $A_m$ is incomparable to $B_{m-1}$. Moreover, we either have $B_m$ and $B_{m-1}$ incomparable, or $B_m \subset B_{m-1}$.
\end{claim1} 
\begin{proof} By construction, $A_m$ and $A_{m-1}\setminus\{i\}$ are part of an $\mathcal N$, and we have one of the two following cases illustrated below, where $K_m,K_m'\in\mathcal F$.
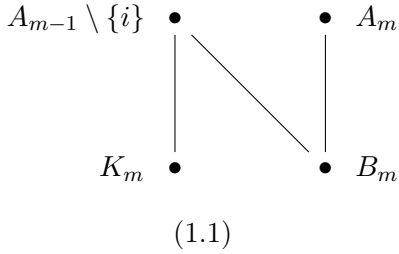
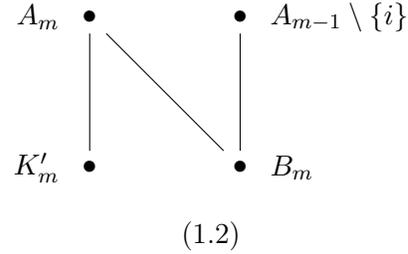
\begin{figure}[H]
\begin{minipage}{0.45\textwidth}
\centering
\begin{tikzpicture}[scale=1]
\node at (0,2) (Si) {\textbf{$\bullet$}};
\node at (2,2) (A) {\textbf{$\bullet$}};
\node at (0,0) (B) {\textbf{$\bullet$}};
\node at (2,0) (C) {\textbf{$\bullet$}};
\node[right=1pt of A] {$A_m$};
\node[left=1pt of B] {$K_m$};
\node[right=1pt of C] {$B_m$};
\node[left=1pt of Si] {$A_{m-1}\setminus\{i\}$};
\draw (A) -- (C);
\draw (B) -- (Si);
\draw (C) -- (Si);
\end{tikzpicture}
\caption*{(1.1)}
\end{minipage}
\begin{minipage}{0.45\textwidth}
\centering
\begin{tikzpicture}[scale=1]
\node at (0,0) (B) {\textbf{$\bullet$}};
\node at (0,2) (A) {\textbf{$\bullet$}};
\node at (2,2) (Si) {\textbf{$\bullet$}};
\node at (2,0) (C) {\textbf{$\bullet$}};
\node[left=1pt of A] {$A_m$};
\node[left=1pt of B] {$K_m'$};
\node[right=1pt of C] {$B_m$};
\node[right=1pt of Si] {$A_{m-1}\setminus\{i\}$};
\draw (A) -- (B);
\draw (Si) -- (C);
\draw (C) -- (A);
\end{tikzpicture}
\caption*{(1.2)}
\end{minipage}
\end{figure}
We show first that we always have $A_m \subset A_{m-1}$. Indeed, this must be the case in the picture on the left, as otherwise $A_{m-1}, A_m, B_m, K_m$ would form an $\mathcal N$ in $\mathcal F$. In the picture on the right, if $A_m \not \subset A_{m-1}$, they are incomparable as we cannot have $A_{m-1}\setminus\{i\}\subset A_{m-1}\subseteq A_m$. Thus, in order for $A_{m-1}, A_m, B_m, K_m'$ to not form an $\mathcal N$ in $\mathcal F$, we must have $K_m' \subset A_{m-1}$. This forms a butterfly, namely $A_{m-1}, A_m, B_m, K_m'$. By Corollary~\ref{cor:butterfly}, there exists $M'\in\mathcal F$ such that $K_m',B_m\subset M'\subset A_m, A_{m-1}$. Since $K_m' \subset A_{m-1}$, but it is incomparable to $A_{m-1}\setminus\{i\}$, we have that $i \in K_m'$, and so $i \in M'$. Since $M' \subset A_{m-1}$, we get that $M'$ is incomparable to $A_{m-1} \setminus\{i\}$. Therefore $M', K_m', B_m, A_{m-1}\setminus\{i\}$ form an $\mathcal N$, contradicting the minimality of $|A_m|$ (replaced by $M'$). Therefore $A_m\subset A_{m-1}$. This also implies that $i\in A_m$.

Next, suppose that $A_m$ and $B_{m-1}$ are comparable. If $B_{m-1}\subseteq A_m\subset A_{m-1}$, we get by Claim 5 for $m-1$ that $B_{m-1}=A_m$, a contradiction as $i\notin B_{m-1}$. For the same reason, we also cannot have $A_m\subset B_{m-1}$. Therefore, $A_m$ and $B_{m-1}$ are incomparable. Finally, we cannot have $B_{m-1}\subseteq B_m$, otherwise $B_{m-1}\subseteq B_m\subset A_m$, a contradiction.
\end{proof}
We now make a couple of useful observations that will help in future analysis. 

\noindent\textbf{Observation 1.} If $A_{m-1}\setminus\{i\}$ is the maximal comparable to both minimal elements of the $\mathcal N$, and $K_m$ is the other minimal (figure 1.1), if $B_m\subset B_{m-1}$, then $B_{m-1}$ and $K_m$ are incomparable. Indeed, $B_{m-1} \not \subseteq K_m$, otherwise $B_m \subset K_m$. Moreover, $K_m \not \subset B_{m-1}$, otherwise $B_{m-1}, K_m, B_m, A_m$ form an $\mathcal N$ in $\mathcal F$.
			
\noindent\textbf{Observation 2.} If $A_{m-1}\setminus\{i\}$ is only comparable to $B_m$, and $K_m'$ is the other minimal of the $\mathcal N$ (figure 1.2), then $B_{m-1}$ and $B_m$ are incomparable. Indeed, if $B_m \subset B_{m-1}$, then $K_m' \subset B_{m-1}$, otherwise $A_m, K'_m, B_m, B_{m-1}$ would form an $\mathcal N$ in $\mathcal F$. But then $K_m' \subset A_{m-1}\setminus \{i\}$, a contradiction.
\begin{claim1}
Let $T \in \cF$. If $B_m \subset T$, then $T$ is comparable to $A_{m-1}$.
\end{claim1}
\begin{proof} Suppose that there exists $T\in\mathcal F$ such that $B_m\subset T$, but $T$ and $A_{m-1}$ are incomparable, in particular $T\not\subseteq B_{m-1}$. By Claim 5 for $m-1$, we cannot have $B_{m-1}\subset T$, so  $T$ and $B_{m-1}$ are incomparable.

If $B_m$ and $B_{m-1}$ are incomparable, then $A_{m-1}, T, B_{m-1}, B_m$ form an $\mathcal N$ in $\mathcal F$, a contradiction.

If $B_m \subset B_{m-1}$, then we are in case (1.1) ($A_{m-1}\setminus\{i\}$ is comparable to both minimal elements). Since $A_{m-1}, K_m, B_m, T$ cannot form an $\mathcal N$, we must have $K_m \subset T$. Since, by Observation 1, $B_{m-1}$ is incomparable to $K_{m}$, we have that $A_{m-1}, B_{m-1}, K_{m}, T$ form an $\mathcal N$, a contradiction.
\end{proof}
\begin{claim1}$B_m \cup \{i\}$ can only be a maximal element of an $\mathcal N$ in $\mathcal F\cup\{B_m\cup\{i\}\}$.
\end{claim1}
\begin{proof}
Suppose that $B_m \cup \{i\} \notin \mathcal{F}$ is a minimal element of an $\mathcal{N}$ in $\mathcal F\cup\{B_m\cup\{i\}\}$. We will find an $\mathcal N$ formed by $A_{m-1}\setminus\{i\}$ which contradicts the maximality of $|B_m|$. We split the analysis into the following natural cases.
			
Case 1. $B_m\cup\{i\}$ is the minimal element comparable to exactly one maximal element, as illustrated below, where $X,Y,Z\in\mathcal F$.
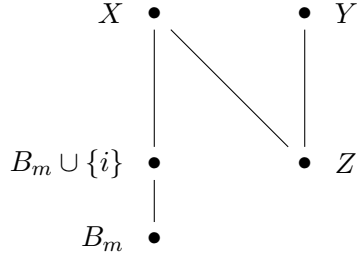
\begin{center}
\begin{tikzpicture}[scale=1]
\node at (0,0) (Si) {\textbf{$\bullet$}};
\node at (0,2) (A) {\textbf{$\bullet$}};
\node at (2,2) (C) {\textbf{$\bullet$}};
\node at (2,0) (B) {\textbf{$\bullet$}};
\node at (0,-1) (S) {\textbf{$\bullet$}};
\node[left=1pt of A] {$X$};
\node[left=1pt of Si] {$B_m \cup \{i\}$};
\node[right=1pt of C] {$Y$};
\node[right=1pt of B] {$Z$};
\node[left=1pt of S] {$B_m$};
\draw (A) -- (B);
\draw (B) -- (C);
\draw (A) -- (Si);
\draw (S) -- (Si);
\end{tikzpicture}
\end{center}
By Lemma~\ref{lem:simplified case analysis2}, we may assume that $B_m \subset Z\subset Y$. Since $Z$ and $Y$ are incomparable to $B_m \cup \{i\}$, we have that $i \notin Z$ and $i \notin Y$. Since $B_m \subset X, Y, Z$, by Claim 2 above, $X, Y, Z$ are all comparable to $A_{m-1}$. Since $i\in A_{m-1}$, we get that $Y,Z\subset A_{m-1}\setminus\{i\}\subset A_{m-1}$. Since $X$ and $Y$ are incomparable, we must also have $X\subset A_{m-1}$. Since $i\in X$, this implies that $X$ and $A_{m-1}\setminus\{i\}$ are incomparable.
			 
By Claim 5 for $m-1$, we have that $B_{m-1} \not \subset X$, otherwise $X$ lies strictly between $B_{m-1}, A_{m-1}$. Also, since $B_{m-1}\subset A_{m-1}\setminus\{i\}$, we cannot have $X\subseteq B_{m-1}$. Therefore $X$ and $B_{m-1}$ are incomparable. Consequently $B_{m-1} \not \subseteq Z$, so $Z$ is either incomparable to $B_{m-1}$, or $Z \subset B_{m-1}$.
			
If $B_{m-1}$ and $Z$ are incomparable, then $A_{m-1}\setminus\{i\}, B_{m-1}, Z, X$ form an $\mathcal N$, contradicting the maximality of $|B_m|$ (replaced by $Z$). Therefore $Z \subset B_{m-1}$, which implies that $B_m \subset B_{m-1}$. This can only happen in the case where $A_{m-1}\setminus\{i\}$ is the maximal element comparable to both minimal elements (figure 1.1). By Observation 1, $B_{m-1} $ and $K_m$ are incomparable, so $K_m \not \subseteq Z$. Moreover, $Z \not \subset K_m$, as $B_m \subset Z$ and $B_m$ and $K_m$ are incomparable. Therefore, $Z$ and $K_m$ are incomparable.
			
Since $X, Z, B_{m-1}, K_m$ cannot form an $\mathcal N$, we must have $K_m$ and $X$ incomparable ($X\not\subseteq K_m$ as $B_m$ and $K_m$ are incomparable). But then $A_{m-1}\setminus\{i\}, K_m, Z, X$ form an $\mathcal N$, contradicting the maximality of $|B_m|$ (replaced by $Z$).
			
Case 2. $B_m\cup\{i\}$ is the minimal element comparable to both maximal elements, as illustrated below, where $X',Y',Z'\in\mathcal F$.
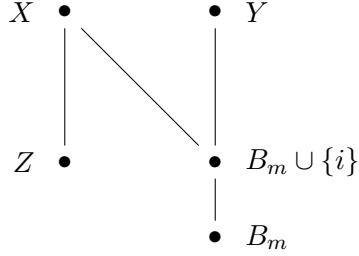
\begin{center}
\begin{tikzpicture}[scale=1]
\node at (0,0) (B) {\textbf{$\bullet$}};
\node at (0,2) (A) {\textbf{$\bullet$}};
\node at (2,2) (C) {\textbf{$\bullet$}};
\node at (2,0) (Si) {\textbf{$\bullet$}};
\node at (2,-1) (S) {\textbf{$\bullet$}};
\node[left=1pt of A] {$X'$};
\node[right=1pt of Si] {$B_m \cup \{i\}$};
\node[right=1pt of C] {$Y'$};
\node[left=1pt of B] {$Z'$};
\node[right=1pt of S] {$B_m$};
\draw (A) -- (B);
\draw (C) -- (Si);
\draw (A) -- (Si);
\draw (S) -- (Si);
\end{tikzpicture}
\end{center}
Again, we may assume that $B_m \subset X', Y', Z'$, which by Claim 2 gives that $X', Y', Z'$ are all comparable to $A_{m-1}$. Since $B_m \subset Z'$, but $B_m \cup \{i\} \not \subset Z'$, we have that $i\notin Z'$, thus $Z'\subset A_{m-1} \setminus\{i\}$. Since $Y'$ is incomparable to $Z'$, we must have that $Y' \subset A_{m-1}$, and since $X'$ is incomparable to $Y'$, we must have that $X' \subset A_{m-1}$. Then, the exact argument as in Case 1 follows. 

More precisely, in the same way, we get that $X'$ is incomparable to both $A_{m-1}\setminus\{i\}$ and $B_{m-1}$, and that $Z'$ is either incomparable to $B_{m-1}$, or $Z'\subset B_{m-1}$. If $B_{m-1}$ and $Z'$ are incomparable, then $A_{m-1}, B_{m-1}, X', Z'$ form an $\mathcal N$, contradicting the maximality of $B_{m}$ (replaced by $Z'$). If $Z'\subset B_{m-1}$, exactly as in Case 1, we get that $Z'$ and $K_m$ are incomparable, and by looking at $K_m, X', Z', B_{m-1}$, we deduce that $K_m$ and $X'$ must also be incomparable. Finally, this means that $A_{m-1}\setminus\{i\}, K_m, Z', X'$ form an $\mathcal N$, contradicting the maximality of $B_m$ (replaced by $Z'$).
\end{proof}
\begin{claim1}
There exists $D_m \in \mathcal F$ such that $M \cup \{i\} \subset D_m \subset B_m \cup \{i\}$.
\end{claim1}
\begin{proof}
By Claim 3, $B_m \cup \{i\}$ must be a maximal element of an $\mathcal N$ in $\mathcal F\cup\{B_m\cup\{i\}\}$. We therefore have two cases, as illustrated below, where $X,Y,Z,X',Y',Z'\in\mathcal F$.
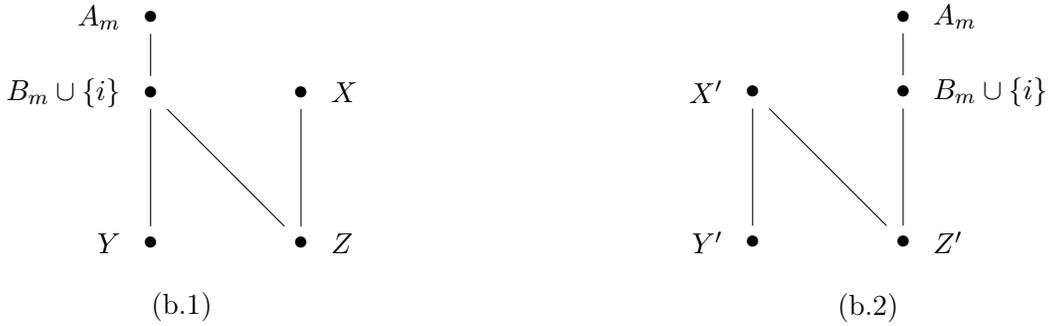
\begin{figure}[H]
\begin{minipage}{0.45\textwidth}
\centering
\begin{tikzpicture}[scale=1]
\node at (0,2) (Si) {\textbf{$\bullet$}};
\node at (2,2) (A) {\textbf{$\bullet$}};
\node at (0,0) (B) {\textbf{$\bullet$}};
\node at (2,0) (C) {\textbf{$\bullet$}};
\node at (0, 3) (AA)  {\textbf{$\bullet$}};
\node[right=1pt of A] {$X$};
\node[left=1pt of B] {$Y$};
\node[right=1pt of C] {$Z$};
\node[left=1pt of Si] {$B_m \cup \{i\}$};
\node[left=1pt of AA] {$A_m$};
\draw (A) -- (C);
\draw (B) -- (Si);
\draw (C) -- (Si);
\draw (AA) -- (Si);		
\end{tikzpicture}
\caption*{(b.1)}
\end{minipage}
\hfill
\begin{minipage}{0.45\textwidth}
\centering
\begin{tikzpicture}[scale=1]
\node at (0,0) (B) {\textbf{$\bullet$}};
\node at (0,2) (A) {\textbf{$\bullet$}};
\node at (2,2) (Si) {\textbf{$\bullet$}};
\node at (2,0) (C) {\textbf{$\bullet$}};
\node at (2,3) (AA) {\textbf{$\bullet$}};
\node[left=1pt of A] {$X'$};
\node[left=1pt of B] {$Y'$};
\node[right=1pt of C] {$Z'$};
\node[right=1pt of Si] {$B_m \cup \{i\}$};
\node[right=1pt of AA] {$A_m$};
\draw (A) -- (B);
\draw (Si) -- (C);
\draw (Si) -- (AA);
\draw (A) -- (C);
\end{tikzpicture}
\caption*{(b.2)}
\end{minipage}
\end{figure}
In the first case, diagram (b.1), if $i \notin Y\cup Z$, then $Y, Z \subseteq B_{m}$. Moreover, $B_m \not \subset X$ as $Y$ is incomparable to $X$, so $Z\neq B_m$. Since $X$ is incomparable to $B_m \cup \{i\}$, we cannot have $X\subseteq B_m$. Therefore $X$ and $B_m$ are incomparable, and so either $Y\neq B_m$, $B_m, X, Y, Z$ form an $\mathcal N$ in $\mathcal F$, a contradiction; or $Y=B_m$, then $Z$ is incomparable to $B_m$ but $Z \subset B_m\cup \{i\}$, so $i\in Z$, a contradiction. 
			
Therefore we must have $i \in Y$ or $i\in Z$, and we take $D_m=Y$ or $D_m= Z$ such that $i\in D_m$. Since $Y, Z \subset A_m \subset A$, all of $X, Y, Z, A_m, A$ are in the same component $\mathcal G$, thus they must be comparable to $M$ by Lemma~\ref{lem:max/min midpoint comparable to everything}. Since $i \notin M$, we must have $M \subset D_m$. This gives $M \cup \{i\} \subset D_m \subset B_m \cup \{i\}$.

In the second case, diagram (b.2), we have that $i\in Z'$. 
			
Indeed, suppose that $i\notin Z'$. Then $Z' \subseteq B_m$. Furthermore, since $Y'$ is incomparable to $Z'$ and $B_m \cup \{i\}$, we must have that $Y'$ and $B_m$ are incomparable. In order for $X', Y', Z', B_m$ to not form an $\mathcal N$, we must have $B_m\subset X'$. Since $X'$ and $B_m\cup\{i\}$ are incomparable, we get that $i \notin X'$.

By Lemma~\ref{lem:simplified case analysis}, we may assume that $X' \subset A_m$. By Claim 1 we also have that $A_m \subset A_{m-1}$, thus $X' \subset A_{m-1}\setminus \{i\}$. Since $A_m$ and $B_{m-1}$ are incomparable (Claim 1), $X'$ cannot be incomparable to $B_{m-1}$, otherwise $A_{m-1}\setminus\{i\}, B_{m-1}, X', A_m$ would form an $\mathcal N$, contradicting the maximality of $|B_m|$ (replaced by $X'$). Since $B_{m-1}$ and $A_m$ are incomparable, we cannot have $B_{m-1} \subseteq X'$ otherwise $B_{m-1} \subset A_m$, thus $X' \subset B_{m-1}$. This implies that $B_m \subset B_{m-1}$, which means that $A_{m-1}\setminus\{i\}$ is the maximal element comparable to both minimal elements in an $\mathcal N$, i.e. we have the picture (1.1) in Claim 1. Putting everything together we get the following diagram.	
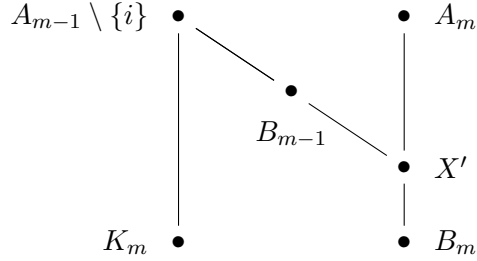
\begin{center}
\begin{tikzpicture}[scale=1]
\node at (0,3) (Si) {\textbf{$\bullet$}};
\node at (3,3) (A) {\textbf{$\bullet$}};
\node at (0,0) (B) {\textbf{$\bullet$}};
\node at (3,0) (C) {\textbf{$\bullet$}};
\node at (1.5,2) (S) {\textbf{$\bullet$}};
\node at (3,1) (X)  {\textbf{$\bullet$}};
\node[right=1pt of A] {$A_m$};
\node[left=1pt of B] {$K_m$};
\node[right=1pt of C] {$B_m$};
\node[left=1pt of Si] {$A_{m-1}\setminus\{i\}$};
\node[below=1pt of S] {$B_{m-1}$};
\node[right=1pt of X] {$X'$};
\draw (B) -- (Si);
\draw (S) -- (Si);
\draw (S) -- (Si);
\draw (A) -- (X);	
\draw (C) -- (X);			
\draw (S) -- (X);			
\end{tikzpicture}
\end{center}
By Observation 1, we know that $K_m$ and $B_{m-1}$ are incomparable, and that $B_{m-1}$ and $A_m$ are incomparable. Since $K_m$ is incomparable to both $B_{m-1}$ and $B_m$, and $B_m\subset X'\subset B_{m-1}$, $X'$ and $K_m$ must be incomparable too. Therefore $A_{m-1}\setminus\{i\}, K_m, X', A_m$ form an $\mathcal N$, contradicting the maximality of $|B_m|$ (replaced by $X'$).
			
Thus, we have $i \in Z'$. Let $D_m= Z'$. Since $Z' \subset A_m \subset A$, we have $D_m \in \mathcal G$ and therefore it is comparable to $M$. As $i\notin M$, we get that $M\subset D_m$, thus $M \cup \{i\} \subset D_m \subset B_m \cup \{i\}$ as desired.
\end{proof}
By the above claim, we have that $D_m\subset A_m$, and since $i\notin B_m$, $D_m$ and $B_m$ are incomparable. Therefore $A_m, B_m, D_m$ form a $\Lambda_2$, or a \textit{chevron}. We will use this fact to prove the next claim.
\begin{claim1} Let $T, T' \in \mathcal F$. If $B_m \subseteq T \subseteq A_m$, then $T = A_m$ or $T = B_m$. If $B_m \subset T'$, then $A_m \subseteq T'$. 
\end{claim1}	
\begin{proof}Suppose that $B_m \subset T \subset A_m$. 

If $A_{m-1}\setminus\{i\}$ is the maximal comparable to both minimal elements (figure (1.1) in Claim 1), since $K_m$ is incomparable to both $A_m$ and $B_m$, we must have that $T$ is incomparable to $K_m$. If $i\in T$, then $A_{m-1}\setminus\{i\}$ and $T$ are incomparable, and so $A_{m-1}\setminus\{i\}, K_m, B_m, T$ form an $\mathcal N$, contradicting the minimality of $|A_m|$ (replaced by $T$). If $i \notin T$, then $T\subset A_m\subset A_{m-1}$, thus $T\subset A_{m-1}\setminus\{i\}$. Therefore $A_{m-1}\setminus\{i\}, K_m, A_m, T$ form an $\mathcal N$, contradicting the maximality of $|B_m|$ (replaced by $T$).
			
Now suppose that $A_{m-1}\setminus\{i\}$ is the maximal comparable to exactly one minimal element (figure 1.2 in Claim 1). If $i \notin T$, then $T\subset A_{m-1}\setminus\{i\}$ is also incomparable to $K_m'$. This is because $T \not \subseteq K_m'$ as $K_m'$ and $B_m$ are incomparable, and $K_m' \not \subset T$, otherwise $K_m '\subset T \subset A_{m-1}\setminus\{i\}$. Therefore $A_{m-1}\setminus\{i\}, K_m', A_m, T$ form an $\mathcal N$, contradicting the maximality of $|B_m|$ (replaced by $T$). If $i \in T$, then $T \not \subseteq B_{m-1}$ as $i\notin B_{m-1}$. Also, since $T\subset A_{m-1}$, $T$ and $A_{m-1}\setminus\{i\}$ are incomparable. Moreover, $T \subset A_m \subset A_{m-1}$, so by Claim 5 with $m-1$, we cannot have $B_{m-1} \subset T$. Therefore $B_{m-1}$ and $T$ are incomparable. As remarked at the end of Claim 1, in this case we have that $B_{m-1}$ and  $B_m$ are incomparable. Therefore $A_{m-1}\setminus\{i\}, B_{m-1}, T, B_m$ form an $\mathcal N$, contradicting the minimality of $|A_m|$ (replaced by $T$).

For the second part, suppose that $B_m \subset T'$, but $A_m \not \subseteq T'$. By the first part of the claim, $T' \not \subset A_m$, therefore $T'$ and $A_m$ are incomparable.
			
We now consider the chevron $A_m, B_m, D_m$, which must not form an $\mathcal N$ with $T'$, thus $D_m \subset T'$. But now $A_m, D_m, B_m,T'$ form a butterfly, so by Corollary~\ref{cor:butterfly}, there exists $M'\in\mathcal F$ strictly between $A_m$ and $B_m$, contradicting the first part of the claim.
\end{proof}

\begin{claim1} $A_m\setminus\{i\}\notin\mathcal F$ can only be a maximal element of an $\mathcal N$ in $\mathcal F\cup\{A_m\setminus\{i\}\}$.
\end{claim1}
\begin{proof}
Suppose that $A_m\setminus\{i\} \notin \mathcal{F}$ is a minimal element of an $\mathcal{N}$ in $\mathcal F\cup\{A_m\setminus\{i\}\}$. Depending on the type of minimal element, we have the following two cases, where $X,Y,Z,X',Y',Z'\in\mathcal F$.
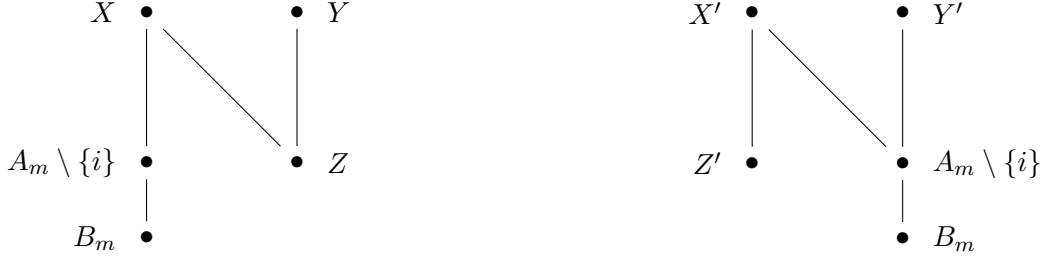
\begin{figure}[H]
\begin{minipage}{0.45\textwidth}
\centering
\begin{tikzpicture}[scale=1]
\node at (0,0) (Si) {\textbf{$\bullet$}};
\node at (0,2) (A) {\textbf{$\bullet$}};
\node at (2,2) (C) {\textbf{$\bullet$}};
\node at (2,0) (B) {\textbf{$\bullet$}};
\node at (0,-1) (S) {\textbf{$\bullet$}};
\node[left=1pt of A] {$X$};
\node[left=1pt of Si] {$A_m\setminus \{i\}$};
\node[right=1pt of C] {$Y$};
\node[right=1pt of B] {$Z$};
\node[left=1pt of S] {$B_m$};
\draw (A) -- (B);
\draw (B) -- (C);
\draw (A) -- (Si);
\draw (S) -- (Si);
\end{tikzpicture}
\end{minipage}
\hfill
\begin{minipage}{0.45\textwidth}
\centering
\begin{tikzpicture}[scale=1]
\node at (0,0) (B) {\textbf{$\bullet$}};
\node at (0,2) (A) {\textbf{$\bullet$}};
\node at (2,2) (C) {\textbf{$\bullet$}};
\node at (2,0) (Si) {\textbf{$\bullet$}};
\node at (2,-1) (S) {\textbf{$\bullet$}};
\node[left=1pt of A] {$X'$};
\node[right=1pt of Si] {$A_m\setminus\{i\}$};
\node[right=1pt of C] {$Y'$};
\node[left=1pt of B] {$Z'$};
\node[right=1pt of S] {$B_m$};
\draw (A) -- (B);
\draw (C) -- (Si);
\draw (A) -- (Si);
\draw (S) -- (Si);
\end{tikzpicture}
\end{minipage}
\end{figure}
By Lemma~\ref{lem:simplified case analysis2}, we may assume that $B_m \subset Z$ in the first case, and $B_m\subset Z'$ in the second case. By Claim 5, we get that $A_m \subseteq Z$, thus $A_m\setminus\{i\} \subset Z$, a contradiction.
\end{proof}
		
This finishes the inductive step. Therefore, this produces an infinite sequence of sets $A_0 = A \supset A_1 \supset \dots \supset M \cup \{i\}$. However, $A_0$ is finite, so the process must terminate, a contradiction.
		
Therefore, there must exist $S\in\mathcal F$ such that $i\notin S$ and $S\cup\{i\}\in\mathcal F$, which finishes the proof.
\end{proof}
\begin{theorem}\label{thm:lower bound on saturation number} Let $\mathcal F$ be an $\mathcal N$-saturated family with ground set $[n]$. Then $|\mathcal F|\geq\frac{n+6}{4}$.\end{theorem}
\begin{proof}
Since we are only interested in the size of $\mathcal F$, we may assume by Lemma~\ref{lem:wlog F satisfies certain conditions} that $\mathcal F$ has a component $\mathcal G$ that contains a set $M$ comparable to all elements of $\mathcal G$ such that $|M|\leq n/2$. Then, by Proposition~\ref{prop:main result}, there are at least $n/2-|\mathcal F|+ 2 $ singletons $i \in [n]$ such that there exist $S_i, S_i \cup \{i\}\in \mathcal F, i \notin S_i$. For each such $i$, we fix a choice of $S_i$ with this property. We now construct a graph $G$ with vertex set $\mathcal F$, and edges only between the pairs $(S_i, S_i\cup \{i\})$, for all such singletons. By construction, $G$ has at least $n/2  -|\mathcal F|+ 2$ edges.

We observe that $G$ is acyclic. Indeed, suppose that there exists a cycle $X_1, \dots, X_{k-1}, X_k \in G$ with $X_k = X_1$, and $(X_i, X_{i+1}) \in E(G)$ for $1\leq i<k$. Let $j_i$ be the singleton corresponding to the edge $(X_i, X_{i+1})$, and since the edges are distinct, the $j_i$'s are also distinct. By construction, $j_1 \in X_1$ if and only if $j_1 \notin X_2$, if and only if $j_1 \notin X_3\dots$ if and only if $j_1 \notin X_k = X_1$, a contradiction.
		
Therefore $G$ is indeed acyclic, and so the number of edges is at most $|\mathcal F| -1$, which implies that $n/2-|\mathcal F|+2\leq |\mathcal F|-1$, hence $|\mathcal F|\geq\frac{n+6}{4}$. 
\end{proof}
Theorem~\ref{thm:lower bound on saturation number}, together with the fact that $\text{sat}^*(n, \mathcal N)\leq 2n$ shows that the saturation number of the $\mathcal N$ poset has indeed linear growth.
\begin{theorem}\label{mainresult2}
$\text{sat}^*(n,\mathcal N)=\Theta(n)$.
\end{theorem}
\section{Concluding remarks and further work}
While the saturation number for the $\mathcal N$ poset is linear, the precise structure of an $\mathcal N$-saturated family remains elusive. Our proof sheds a bit of light on this question via Lemma~\ref{lem:complete bip graph has a midpoint}, where we show that in any such family, not just minimal, two antichains, one completely above the other, must have a `midpoint' between them that lies in the family. Moreover, Lemma~\ref{lem:max/min midpoint comparable to everything} tells us that in every connected component of the Hasse diagram of an $\mathcal N$-saturated family, there exists a point comparable to every other set of the component. Of course, these features are far from sufficient to characterise $\mathcal N$-saturated families, since for example any chain has both of these properties. We therefore ask the following.
\begin{question}
Does there exist a structural classification of an $\mathcal N$-saturated family? In other words, how does the Hasse diagram of such a family look like?
\end{question}
Since our proof worked with an arbitrary $\mathcal N$-saturated family, it is plausible that if one were to work with a \textit{minimal} $\mathcal N$-saturated family, i.e. one such that $|\mathcal F|=\text{sat}^*(n,\mathcal N)$, then more structural properties will emerge, and so we also ask the following.
\begin{question}
Let $\mathcal F$ be an $\mathcal N$-saturated family such that $|\mathcal F|=\text{sat}^*(n,\mathcal N)$. How does the Hasse diagram of $\mathcal F$ look like?
\end{question}
It is worth mentioning that while finding the rate of growth of the saturation number is challenging in itself, characterising the saturated families is almost uncharted territory, with the only exception being the antichain, where saturated families are essentially unions of disjoint chains \cite{bastide2024exact}.

Finally, one can ask what is the leading constant for $\text{sat}^*(n,\mathcal N)$. We believe that although our proof gives linearity, it does not give the optimal constant. In fact, we conjecture the following.
\begin{conjecture}
$\text{sat}^*(n,\mathcal N)=2n$.
\end{conjecture}
\noindent\textbf{Acknowledgment.} We would like to thank the anonymous referees for a thorough read and helpful suggestions that have greatly helped the readability of the paper.
\bibliographystyle{amsplain}
\bibliography{references}

@article{ivan2020saturationbutterflyposet,
  title={Saturation for the {B}utterfly {P}oset},
  author={Ivan, Maria-Romina},
  journal={Mathematika},
  volume={66},
  number={3},
  pages={806--817},
  year={2020},
  publisher={Wiley Online Library}
}

@article{ferrara2017saturation,
  title={The {S}aturation {N}umber of {I}nduced {S}ubposets of the {B}oolean lattice},
  author={Ferrara, Michael and Kay, Bill and Kramer, Lucas and Martin, Ryan R and Reiniger, Benjamin and Smith, Heather C and Sullivan, Eric},
  journal={Discrete Mathematics},
  volume={340},
  number={10},
  pages={2479--2487},
  year={2017},
  publisher={Elsevier}
}

@article{keszegh2021induced,
  title={Induced and non-induced poset saturation problems},
  author={Keszegh, Bal{\'a}zs and Lemons, Nathan and Martin, Ryan R and P{\'a}lv{\"o}lgyi, D{\"o}m{\"o}t{\"o}r and Patk{\'o}s, Bal{\'a}zs},
  journal={Journal of Combinatorial Theory, Series A},
  volume={184},
  pages={105497},
  year={2021},
  publisher={Elsevier}
}

@article{freschi2023induced,
  title={The induced saturation problem for posets},
  author={Freschi, Andrea and Piga, Sim{\'o}n and Sharifzadeh, Maryam and Treglown, Andrew},
  journal={Combinatorial Theory},
  volume={3},
  number={3},
  year={2023},
  publisher={eScholarship Publishing}
}

@article{bastide2024exact,
  title={Exact antichain saturation numbers via a generalisation of a result of {L}ehman-{R}on},
  author={Bastide, Paul and Groenland, Carla and Jacob, Hugo and Johnston, Tom},
  journal={Combinatorial Theory},
volume={4},
  year={2024}
}

@article{polynomial,
title={A {P}olynomial {U}pper {B}ound for {P}oset {S}aturation},
author={Paul Bastide and Carla Groenland and Maria-Romina Ivan and Tom Johnston},
journal={European Journal of Combinatorics},
year={2024}
}

@article{gluing,
title={Gluing {P}osets and the {D}ichotomy of {P}oset {S}aturation {N}umbers},
author={Maria-Romina Ivan and Sean Jaffe},
year={2025},
journal={arXiv:2503.12223}
}

@article{diamondlinear,
title={The {S}aturation {N}umber for the {D}iamond is {L}inear},
author={Maria-Romina Ivan and Sean Jaffe},
year={2026},
journal={Bulletin of the London Mathematicsl Society},
volume={58},
pages={e70305}
}
\Addresses
\end{document}